\documentclass[11pt,centertags,reqno]{amsart}

\usepackage[foot]{amsaddr}
\usepackage{latexsym}
\usepackage[english]{babel}
\usepackage[T1]{fontenc}
\usepackage[utf8]{inputenc}
\usepackage[numbers]{natbib}
\usepackage{amssymb}
\usepackage{fancyhdr}
\usepackage{url}
\usepackage{hyperref}
\usepackage{verbatim}
\usepackage{leftidx}
\usepackage{color,graphicx}

\usepackage{mathrsfs, mathtools}
\usepackage{stmaryrd}

\usepackage{marvosym}
\mathtoolsset{showonlyrefs}

\usepackage{soul}

\usepackage{enumerate} 
\usepackage{enumitem} 
\usepackage{bm} 

\numberwithin{equation}{section} \makeatletter
\renewcommand{\subsection}{\@startsection
{subsection}{2}{0mm}{\baselineskip}{0.25cm}
{\normalfont\normalsize\bf}} \makeatother

\newtheorem{theorem}{Theorem}[section]
\newtheorem{lemma}[theorem]{Lemma}
\newtheorem{corollary}[theorem]{Corollary}

\newtheorem{proposition}[theorem]{Proposition}
\newtheorem{ass}[theorem]{Assumption}
\theoremstyle{remark}
\newtheorem{remark}[theorem]{Remark}

\def \F {\mathcal F}

\def \P {\mathbf P}

\def \I {{\mathbf 1}}

\def \R {\mathbb R}
\def \bF {\mathbb F}

\newcommand{\ud}{\mathrm d}

\newcommand{\esp}[2][\mathbb E] {#1\left[#2\right]}
\newcommand{\espt}[2][\mathbb E_t] {#1\left[#2\right]}

\newcommand{\bgamma}{\bm{\gamma}}
\newcommand{\btau}{\bm{\tau}}
\newcommand{\bq}{\bm{q}}
\newcommand{\CO}{\rm CO_2}

\begin{document}
	
	\author[K.~Colaneri]{Katia Colaneri}\address{Katia Colaneri, Department of Economics and Finance, University of Rome Tor Vergata, Via Columbia 2, 00133 Rome, Italy.}\email{katia.colaneri@uniroma2.it}

	\author[R.~Frey]{R\"{u}diger Frey}\address{R\"{u}diger Frey, Institute for Statistics and Mathematics, Vienna University of Economics and Business, Welthandelpatz 1, 1020 Vienna, Austria}\email{rfrey@wu.ac.at}

	\author[V.~K\"{o}ck]{Verena K\"{o}ck}\address{Verena K\"{o}ck, Institute for Statistics and Mathematics, Vienna University of Economics and Business, Welthandelpatz 1, 1020 Vienna, Austria}\email{verena.koeck@wu.ac.at}

\title[]{Random carbon tax policy and investment into emission abatement technologies}

\date{\today}
	
\begin{abstract}
We study the problem of a profit maximizing electricity producer who has to pay carbon taxes and who decides on investments into  technologies for the  abatement of $\CO$ emissions in an environment where   carbon tax policy is random and where  the investment in the abatement technology is divisible, irreversible and subject to transaction costs.
We consider two approaches for modelling the randomness in taxes. First we assume a precise probabilistic model for the tax process, namely a pure jump Markov process (so-called tax risk); this leads to a stochastic control  problem for the investment strategy.  
Second, we  analyze the case of  an {uncertainty-averse} producer who uses a differential game to decide on optimal production and investment.   We carry out a rigorous mathematical analysis of the producer's optimization problem and of the associated nonlinear PDEs in both cases. Numerical methods are used  to study quantitative properties of the optimal investment strategy.   
We find that in the tax risk case   the   investment  in abatement technologies is typically lower  than in a benchmark  scenario with  deterministic taxes. However, there are a couple of interesting new twists related to  production technology, divisibility of the investment, tax rebates and investor expectations.   In the stochastic differential game on the other hand an increase in uncertainty  might stipulate more investment.
\end{abstract}

\maketitle



%


\date{\today}

\section{Introduction}

Carbon taxes and trading of emission certificates are commonly considered key policy tools for reducing carbon pollution  and hence for mitigating  climate change. Academic contributions in this field from an environmental economic perspective have mainly focused on  
\emph{optimal} tax schemes or optimal carbon prices for an efficient emission reduction, see for instance the seminal contributions  by 
\citet{nordhaus1993Dice, nordhaus2019climate, golosov2014optimal, acemoglu2012environment}. 
More recently this problem has been addressed within the literature on continuous-time stochastic control by, e.g., \citet{bib:aid-biagini-23}, \citet{bib:aid-kemper-touzi-23}, or \citet{bib:carmona-et-al-21} (these papers are discussed in Section \ref{sec:literature}). 
While the design of an optimal tax scheme or carbon price is a very  relevant research question, in reality  emission tax policy is affected by many unpredictable factors such as changes in political sentiment and  election results, lobbying by industry groups, or developments in international climate policy.
Therefore  future tax rates are random and long-term emission tax schemes announced by governments are not  fully credible from the viewpoint of carbon emitting producers. This is a prime example of the so-called \emph{climate policy uncertainty}. In environmental  economics it is often  argued that policy uncertainty  has a negative impact on investments in  carbon abatement technology. For instance the British newspaper The Economist \cite{bib:economist-23} writes 
\begin{quote}
     Political polarisation [regarding the relevance of climate change] means bigger flip-flops when power changes hands: imagine France under the wind-farm-loathing Marine Le Pen. Everywhere, making climate policy less predictable makes it harder for investors to plan for the long term, as they must.
\end{quote} 
 On the  policy side, a report by the International Energy Agency (\citet{bib:IEA-08}) discusses the  impact of policy-induced jumps in  carbon prices  on the incentives for investing in low-carbon power-generation technologies. The report argues   that 
``the greater the level of policy uncertainty, the less effective climate change policies will be at incentivising investment in low-emitting technologies.''
The empirical study  of   \citet{berestycki2022measuring}  develops indices for climate policy uncertainty  that are based on newspaper coverage, and they carry out  a regression analysis showing  that higher levels of uncertainty are associated with a substantial decrease in the  investment level in carbon intensive industries.

\citet{bib:fuss-et-al-08} and of  \citet{bib:yang-et-al-07} (this latter paper is related to the IEA-report of \citet{bib:IEA-08}) propose  formal models for analysing the implications of randomness in carbon price policy (which is economically very similar to a carbon tax policy) on investment in carbon capture and storage technologies.  To the best of our knowledge, these are the only contributions of this type so far. 
In these works there are two sources of randomness.   First, there are fluctuations in  the electricity and in the carbon price. Second, there is  randomness in the carbon pricing policy itself:   at a deterministic future time point $\bar t$ government announces if the  carbon pricing scheme continues or if  the policy is abolished, where the probability of both events is known. \citet{bib:fuss-et-al-08} and \citet{bib:yang-et-al-07} use a real options approach where investment is immediate (no construction times), indivisible, and irreversible. In that setup it is optimal to  postpone  investment decisions until the announcement  date $\bar t$, so that randomness in policy delays investment into abatement technology. Fluctuations in market prices, on the other hand, have very little impact on the investment decision.    From a technical point of view,  \cite{bib:fuss-et-al-08} and \cite{bib:yang-et-al-07} consider a discrete time settings and solve the optimization  via numerical methods based on Monte-Carlo simulation and backward induction.

In our paper we go beyond the analysis of \cite{bib:fuss-et-al-08} and \cite{bib:yang-et-al-07}. We study the problem of a profit maximizing electricity producer who pays taxes on emissions and may invest into emission abatement technology in a continuous time framework that incorporates  a rich set of approaches for modeling randomness in carbon tax policy. 
In our setup  investments are  divisible as,  for instance, in the installation of new solar panels, and the  producer chooses the  \emph{rate} $\bgamma $ at which she invests. 
Moreover, investment  is irreversible (i.e.~$\gamma_t \ge 0$)  and subject to  {transaction costs} that prohibit  a  rapid  adjustment   of  the investment level. This framework includes stylized forms of emission abatement technology, such as the case where producer has the option to retrofit existing gas-fired power plants with a carbon capture and storage or \emph{filter} technology, and the case where she may reduce the costs for producing electricity  by investing into a novel green technology with lower marginal production costs. We discuss these examples in detail in the theoretical part of the paper and in our numerical experiments.

To deal with randomness in  carbon taxes we consider two different approaches. First, we assume a precise probabilistic model for the tax process, namely a pure jump Markov process.  In decision-theoretic terms this corresponds to the paradigm of {risk}, so that we refer to this situation as  \emph{tax risk}. In the tax risk  case  the producer is confronted with a   stochastic control problem with the investment rate  as control variable. 
Second, we  analyze the case of   an \emph{uncertainty-averse} producer  who  considers a set  of possible future tax scenarios but  does not postulate a probabilistic model for the tax evolution. Instead she determines her  production and investment policy as  equilibrium strategy of a  game with a malevolent opponent. The objective of the producer remains that of maximizing expected profits, whereas the opponent chooses a tax process from the set of scenarios  to minimize the  profits of the producer. 
In both cases, tax risk and tax uncertainty,  we carry out a precise mathematical analysis of the producer's optimization problem. For the tax risk case we  characterize the value function  as unique {viscosity solution} of the  associated HJB equation, using general results from \citet{pham1998optimal}.  Moreover  we give conditions for the existence of classical solutions. For the tax uncertainty setup we end up with a differential game for which we establish existence of an equilibrium and we characterise the value of the game in terms of a classical solution of the Bellman--Isaacs equation.   Since explicit solutions to these equations 
exist only in very special situations, we conduct numerical experiments to analyze the investment behaviour of the producer. 

In our numerical experiments within the paradigm of tax risk we consider two models for the tax dynamics: in the first model, the government may raise taxes at some random future time point, for instance to comply with international climate treaties; in the second model high taxes might be  reversed, for instance since a government with a ``brown'' policy agenda replaces a ``green'' one. The latter situation is somewhat  similar, in spirit, to the framework of \cite{bib:fuss-et-al-08}. We  compare the investment decisions of the producer to a benchmark case where taxes are deterministic.
Our experiments show  that under tax risk the firm is typically less willing to invest into  abatement technologies than in the corresponding benchmark  scenario, which supports the intuition that randomness in carbon taxes may be detrimental for climate policy.   However, there are some new interesting twists. To begin with,  in our setup the producer invests already before a tax increase is actually implemented in order to hedge against high future tax payments. This hedging behavior is not observed in real options models such as \cite{bib:fuss-et-al-08}, where randomness in tax policy induces the producer to delay investment and wait until the government decision. 
In addition, we study the implications  of introducing  an emission-independent tax rebate and we find that rebates enhance the investment into abatement technology. Finally,  our experiments show  that investor expectations are crucial determinants  for the success of a carbon tax policy. In particular, a tax policy  that is not credible (i.e.~producers  are  not convinced that an announced tax increase will actually be implemented or they expect  that a high tax regime will be reversed soon) is substantially less effective than a credible policy.


We go on and study the optimal investment within the stochastic differential game for the uncertainty averse producer. 
Interestingly, in that case the results are reversed, that is, more uncertainty is beneficial from a societal point of view, as it leads to higher investment in carbon abatement technology. Moreover,  a rebate now generally reduces investment. These are interesting new results which show that the paradigm used to model the decision making process of the producer is crucial for the impact of climate policy uncertainty on investment into abatement technology.


The remainder of the paper is organized as follows:  In Section~\ref{sec:literature} we discuss  some of the related literature; in Section~\ref{sec:optim-problem} we introduce the setup and the optimization problem of the electricity producer; in  Section~\ref{sec:examples} we discuss specific examples for the electricity production and emission abatement technology;  Section~\ref{sec:mathematics} is concerned with the  control problem of the producer under tax  risk, whereas the stochastic differential game related to  tax policy uncertainty is studied in Section \ref{sec:game}; in Section~\ref{sec:numerics} we present  the results from numerical experiments  and discuss their economic implications; Section~\ref{sec:conclusion} concludes.


\subsection{Literature review}\label{sec:literature}
We continue with a brief discussion of related contributions.  Within the framework of 
stochastic control literature on optimal tax- and carbon pricing schemes, \citet{bib:aid-biagini-23} study an optimal dynamic carbon emission regulation for a set of firms, in presence of a regulator who may choose dynamically the emission allowances to each firm. The problem is formulated as a Stackelberg game between the regulator and the firms in a jump diffusion setup with linear quadratic costs. This formulation  allows for a closed-form expression of the optimal dynamic allocation policies. 
\citet{bib:aid-kemper-touzi-23} investigate 
the optimal regulatory incentives that trigger the development of green electricity production in a monopoly and in a duopoly setup. The regulator wishes to encourage green investments to limit carbon emissions, while simultaneously reducing the intermittency of the total energy production. Their main results is a characterization of the regulatory contract that naturally includes interesting agreements like rebate. \citet{bib:carmona-et-al-21} analyse mean field control and mean field game models of electricity producers who can decide on the composition of their energy mix (brown or green)  in the presence of a carbon tax. The producers have to balance the cost of intermittency and the amount of carbon tax they pay.  Initially  producers choose their investment in green production technology;  given this choice  they continuously  adjust their usage of fossil fuels and  hence emissions to minimize a given cost function. 
The paper analyses competitive (Nash equilibrium) and cooperative (social optimum) solutions to this problem  via  systems of forward-backward SDEs. It also includes a study of a  Stackelberg game between a regulator who sets the carbon tax rate at the initial date and the mean field of producers.

\citet{bib:tankov-lavigne-23} and \citet{bib:dimitrescu-leutscher-tankov-24} have done  interesting theoretical research  on the implication of randomness in climate policy more generally.   \cite{bib:tankov-lavigne-23} consider a mean-field game model for a large financial market where firms determine their dynamic emission strategies under climate transition risk in the presence of  green and neutral investors. They show among others that uncertainty about future climate policies leads to overall higher emissions in equilibrium. In a similar spirit, \cite{bib:dimitrescu-leutscher-tankov-24} study the impact of transition scenario uncertainty on the pace of decarbonization and on output  prices in  the electricity industry.   Empirical studies on the impact of carbon taxes to production and investment in green technologies include \citet{aghion2016carbon} who studied in particular the effects of taxes and fuel prices on investment in  technological innovation using data for the automobile sector, and \citet{martinsson2024effect} where data on CO2 emissions from Swedish manufacturing sector is used to estimate the impact of carbon pricing on firm-level emission intensities. These studies provide a support that taxes and high prices trigger investment in low emission technology.

\section{The optimization problem of the electricity producer}\label{sec:optim-problem}

Throughout the paper  we fix a horizon date $T$ and a probability space $(\Omega, \F, \P)$ with  filtration $\bF =(\mathcal{F}_t)_{0\le t \le T}$ representing the information flow. In the sequel all processes are assumed to be $\bF$-adapted and expectations are taken with respect to the probability measure  $\P$.

We consider  a profit maximizing   electricity  producer  who has to pay carbon taxes and who may invest in technology for the abatement of $\CO$ emissions.  We denote by $\tau_t$ the  taxes  per unit of emission and by $X_t$ the value of the investment in abatement technology  at time $t$.  The producer chooses the amount  of electricity to be produced at every point in time and she  controls the  process $X = (X_t)_{0\le t \le T}$  by  her investments into abatement technology.  

The price  of electricity and  the production cost 
may be modulated by an exogenous $d$-dimensional factor process $Y = (Y_t)_{0 \le t \le T}$.  We assume that $Y$ follows a $d$-dimensional diffusion,
\begin{align}\label{equation:Y}
\ud Y_t= \beta (t, Y_t) \ \ud t + \alpha(t, Y_t) \ \ud B_t, \quad Y_0=y\in \mathbb{R},
\end{align}
where $B$ is a $d$-dimensional Brownian motion, and where  the drift $\beta (t, y) \in \mathbb R^d$  and the dispersion $\alpha(t, y)\in \mathbb R^{d \times d}$, for $(t, y)\in [0,T]\times \mathbb{R}^d$, satisfy standard conditions for existence and uniqueness of the SDE \eqref{equation:Y}. Moreover, we denote the generator of $Y$ by $\mathcal{L}^Y$, which reads as 
\[
\mathcal L^Y f(y)=\sum_{i=1}^d  f_{y_i} (y) \beta_i (t,y) + \frac{1}{2} \sum_{i,j=1}^d  f_{y_i y_j} (y) \mathfrak{S}_{ij}(t, y), 
\]
where $\mathfrak{S}(t,y)=\alpha(t,y) \cdot \alpha^{\top}(t,y)$.

\subsection{Instantaneous electricity production} 

We assume that the electricity  market is perfectly competitive so that the  producer acts as a price taker, that is she takes the price $p_t=p(Y_t)$ of one unit of electricity as given and adjusts the  quantity produced in order to maximize instantaneous profits (In the numeric examples we also  consider the  case where the quantity to be produced is  fixed.). This situation might arise in the context of a merit order system, where the electricity spot price is determined by the short run marginal production cost of the power plant that is on the margin of the electricity production system.
For a given investment value $x$,  tax rate $\tau$ and value $y$ of the factor process, we denote the cost of producing $q$ units of electricity by $C(q,x,y,\tau)$. Hence the instantaneous profit 
is given by
\begin{align}\label{eq:pi}
 \Pi(q, x, y, \tau)&=p(y)q - C(q, x, y, \tau)+\nu_0(q) \tau\,.
\end{align}
The term  $\nu_0(q) \tau$ models  a \emph{tax rebate} that depends on the amount $q$ of energy produced and on the current tax rate, but not on the actual emissions of the producer.  Tax rebates of this form  penalize  (reward) producers with high (low) emissions compared to the industry average and are part of many proposals for carbon taxes. 
The producer  chooses the  production to maximize her instantaneous profit and we denote the  maximal  profit by
\begin{equation} \label{eq:pi_star}
\Pi^*(x, y, \tau)=\max_{q \in [0,q^{\max}]}\Pi(q, x, y, \tau),
\end{equation}
where the constant $q^{\max} >0$ denotes the maximum   capacity of the production technology. 

The  next assumption gives conditions ensuring that the function $\Pi^*$  is well-defined and enjoys certain regularity properties.  
\begin{ass}\label{ass:cost_regularity}\hfill
\begin{itemize} 
\item[(i)] There are  functions $C_0, C_1 \colon [0,q^{\max}] \times \R \times \R^d \to [0, \infty)$ such that  
   $$ C(q, x, y, \tau)=C_0(q, x, y)+C_1(q, x, y) \tau\,;$$ 
 moreover,  $C_0$ and $C_1$ are   increasing,  strictly convex and $\mathcal{C}^1$ in $q$ and $C_1$ is bounded. \item[(ii)] $C_0$ and $C_1$ are Lipschitz continuous in $x,y$ uniformly in $q \in[0,q^{\max}]$.
\item[(iii)] The function $\nu_0$ is differentiable, increasing and concave on $[0, q^{\max}]$
\item[(iv)] The function $p \colon \R^d \to \R$ is Lipschitz continuous in $y$. 
 \end{itemize}  
\end{ass}
In economic terms the function $C_1$ measures  the emissions from producing $q$ units of electricity given the investment level $x$ and the value of the factor process; these are then multiplied with $\tau$  to give the instantaneous  carbon tax payments; the function $C_0$ gives the emission-independent production cost. Specific examples  are discussed in Section~\ref{sec:examples} below.

Under Assumption~\ref{ass:cost_regularity}(i) and (iii) there is a unique optimal  instantaneous energy output $q^* \in[0,q^{\max}]$ for every $(x,y,\tau)$. Taking derivatives with respect to $q$ gives  the first order condition
\begin{equation}\label{eq:FOC-q}
p(y) - \partial_q C_0(q,x,y) - (\partial_q C_1(q,x,y)  -\partial_q \nu_0(q) )\tau =0, 
\end{equation}
which has at most one solution  due to the strict convexity of the cost functions. If we consider moreover the boundary cases $q=0$ and $q= q^{\max}$ we  get
\begin{equation} \label{eq:cases}
    q^* = 
    \begin{cases} 
    0\,, &\text{if } p(y) - \partial_q C_0(0,x,y) - (\partial_q C_1(0,x,y)  -\partial_q \nu_0(0) )\tau < 0;\\
    q^{\max}, &\text{if } p(y) - \partial_q C_0(q^{\max},x,y) - (\partial_q C_1(q^{\max},x,y)  -\partial_q \nu_0(q^{\max}) )\tau  > 0;\\
    \text{the solution of \eqref{eq:FOC-q}},&\text{else.}
    \end{cases}
\end{equation}

\subsection{Optimal investment problem} 

We assume that  the investment value $X$ has dynamics
\begin{equation}\label{eq:stochastic_investment}
X_t=X_0+ \int_0^t \gamma_s \ud s - \int_0^t \delta X_s \ud s + \sigma W_t, \quad t \le T.  \end{equation}
Here $W=(W_t)_{t \ge 0}$ is a Brownian motion, independent of the $d$-dimensional Brownian motoin $B$, $0\le \delta<1$ is  the depreciation rate and the  term $\sigma W_t$ models exogenous fluctuations of the  investment value, due, for example, to random replacement costs.
The process  $\bgamma =(\gamma_t)_{0 \le t \le T}$ represents  the rate at which the producer invests into  abatement technology. We  assume that the investment is \emph{irreversible}, that is  we introduce  the constraint that $\gamma_t \ge 0$ for all $t$. 
Moreover, we assume that the investment  is subject to proportional \emph{transaction} costs  given by $\kappa \gamma^2$. These costs penalize a rapid build-up of abatement technology. By $\mathcal{A}$ we denote the set of all \emph{admissible} investment strategies, that is the set of  all adapted nonnegative càdlàg processes $\bgamma$ with $\mathbb E\left[\int_0^T \gamma_t \,\ud t \right]< \infty$.

The producer uses a cash account  $D$ to finance her investments and to invest the profits from selling electricity.   We assume that there is a constant interest rate $r\ge0$ that applies to borrowing and lending. Hence $D$ has the dynamics
\begin{equation}\label{eq:cash}
\ud D_t= \big( r D_t + \Pi^*(X_t, \tau_t, Y_t) - (\gamma_t +\kappa \gamma_t^2)\big ) \ud t, \quad D_0=0.
\end{equation}
We interpret the horizon date $T$ as lifetime of the electricity production technology, and we  model the residual value of the investment by a function  $h(X_T)$, which is nonnegative, increasing and continuous and whose form will depend on the type of abatement technology.

The goal of the electricity producer is to maximize $ \esp{e^{-r(T)} (D_T + h(X_T))}$,  the expected discounted value of her terminal cash position and of the residual investment. 
Next we  show that the initial value of the cash account does not affect the investment decision of the producer.  In fact, 
using \eqref{eq:cash} we get that 
\[
e^{-r(T-t)} D_T = D_t - \int_t^T r e^{-r(s-t)}D_s \ud s + \int_t^T e^{-r(s-t)} (r D_s + \Pi^*(X_s, Y_s, \tau_s)-\gamma_s - \kappa \gamma_s^2) \ud s.
\]
Hence 
\begin{equation} \label{eq:optim-prob-producer}
\esp{e^{-r(T)} (D_T + h(X_T))} = D_0 + \esp{\int_0^T \!\! e^{-rs} \left( \Pi^*(X_s, \tau_s, Y_s)-\gamma_s-\kappa \gamma_s^2\right) \ud s +    e^{- r T} h(X_T)}, 
\end{equation}
and the goal of the producer amounts  to maximizing  the second term in \eqref{eq:optim-prob-producer}.   

In this paper we study the optimization problem of the producer in two  settings that differ with respect to the modelling of randomness in carbon taxes. In Section~\ref{sec:mathematics}  we  analyze the  case of tax risk, where  the tax process $\btau$ follows a precise probabilistic model, namely a pure jump process with given jump intensity and jump size distribution. In Section~\ref{sec:game}  we consider an alternative  approach  corresponding to the case of tax uncertainty. There 
the electricity producer considers a set of possible future tax scenarios and uses a worst case approach based on stochastic differential games to determine her investment strategy.

\section{Production technologies: examples }\label{sec:examples}

We now discuss two specific examples  for the production and emission abatement  technology that will be used in the numerical experiments.

\subsection{The Filter Technology}\label{subsec:filter_technology}
In this example we assume that the producer is using a brown technology such as coal fired power plants but is able  to reduce $\CO$ pollution  by investing in a carbon capture and storage or filter technology.
We let $\zeta$ be the input good, i.e. the amount of raw material  (coal or gas) which is needed to produce electricity. 
We suppose that one unit of raw material has a cost of $\bar{c}(y)$ dollars and that for each unit of raw material used in the production process, the amount of emitted  $\CO$  is $e_0$. If filters are installed,  emissions per unit of raw material are reduced by $e(x)$. The emission reduction depends clearly on the quality and the number of filters, and hence on the investment level $x$. 
Given an investment level $x$, total emissions for $\zeta$ units of raw material are thus given by $\zeta(e_0-e(x))$. 

We denote by $P(\zeta)$ the amount of electricity that can be produced using $\zeta$ units of raw material, for a continuous increasing and concave function $P$ with $P(0)=0$.
Denote by $Q(\cdot)$ the inverse function of $P$. Then, to produce the amount $q$ of electricity the producer needs $\zeta=Q(q)$ units of raw material and hence the incurred cost (production cost and taxes) is given by
\begin{align}\label{eq:cost_funct_2tech}
C(q,x, y, \tau)=Q(q)(\bar{c}(y)+(e_0-e(x))\tau).
\end{align}
Note that in this example  the functions $C_0$ and $C_1$ from Assumption~\ref{ass:cost_regularity} are given by 
$C_0(q,x,y)  =Q(q)\bar{c}(y)$ and $C_1(q,x,y) = Q(q)(e_0-e(x))$. 
Recall that we interpret the horizon date $T$ as lifetime of the brown power plant. It makes sense to assume that the residual value of  the filters installed  is zero once the power plant is no longer in operation, so that for the filter technology we take $h(X_T) =0$.

\subsection{Two technologies}\label{subsec:two_technology} 

Next we consider a situation where the energy producer has the option to replace a brown technology such as  coal or gas power plants by a   green technology such as wind, or solar energy. We denote by $\zeta_b$ be the amount of input material for  the brown technology and suppose that one  unit of input material costs $c_b(y)$ dollars and leads to  $e_b$ tons of $\CO$.  Let $P_b(\zeta)$ be the amount of electricity that can be produced with $\zeta$ units of raw material and assume that $P_b$ is increasing and concave and $P_b(0)=0$.

The input material to produce green, on the other hand, has a prize of zero (for instance  wind or sun) and for simplicity we assume that green technology does not emit $\CO$.
We associate the investment level $x$ with the amount of green production facilities (solar panels or wind turbines) installed and we denote by  $P_g(x)$ be the maximum amount of electricity that can be produced with the green technology for a given investment level $x$.
We assume that the maintenance cost $c_g(x)$ of the green technology  only depends of the investment $x$. Denote by $Q_b(\cdot)$ the inverse function of $P_b$.Then the total cost for producing $q$ units of energy is
\begin{equation}
C(q,x, y, \tau)=
\begin{cases}
    c_g(x) \qquad &\mbox{ if } \ q-P_g(x)\le 0\,;\\
    c_g(x) + (c_b(y)+e_b\tau)Q_b\left(q-P_g(x)\right) \qquad &\mbox{ if } \ q-P_g(x)> 0\,;
\end{cases}
\end{equation}
equivalently, $C(q,x, y, \tau)=c_g(x) + (c_b(y)+e_b\tau) Q_b\left(\left(q-P_g(x)\right)^+\right)$. 
In this example we get that 
\begin{align}
C_0(x, y, q)&= c_g(x) + c_b(y) Q_b\left(\left(q-P_g(x)\right)^+\right)\\
C_1(x, y, q) &= e_b Q_b\left(\left(q-P_g(x)\right)^+\right)
\end{align}
which satisfy the regularity conditions stated in Assumption \ref{ass:cost_regularity} (i)--(iii) if $c_g$ and  $P_g$ are Lipschitz
$Q_b$ is $\mathcal{C}^1$, increasing, strictly convex and $(Q_b)'(0)=0$, with the exception that $C_0$ is strictly convex in $q$ only for $q \ge P_g(x)$. This is however sufficient for the existence of a unique optimal electricity output $q^*$ which is given by \eqref{eq:cases}. In the numerical experiments   we take $Q_b(q)=a q^{3/2}$, which fits the conditions above.   
We will also assume that $c_g$ and  $P_g$ are increasing in the investment level $x$ to make the  model reasonable from an economic viewpoint.  
In the  numerical experiments we use a function $P_g$ of the following form
$$P_g(x)=p_g [(x-\bar x)^+]^\alpha, \quad \alpha \in (0,1),$$
for some productivity parameter $p_g \in (0,1)$, where 
$\bar x$ represents initial expenses such as land acquisition and infrastructure development for connecting to the grid that the electricity company must bear  when building a green power plant. 
This example shows that our model may account for threshold effects, even when the investment occurs continuously in time. 
To simplify the exposition we concentrate on  the case where the maximum production level $q^{\max}$ is independent of $x$. In the  two-technology example this bound can be interpreted as the maximum amount of electricity  that could be absorbed by the grid. However, it might also make sense to consider the case where $q^{\max}$  depends on the amount of green technology installed and  hence on the investment level, that is, to model the maximum capacity as a  function $\bar{q}(x)$.   This choice brings additional technicalities in some of the theoretical results. We refer, to the comment on maximum capacity expansion in Appendix \ref{appendix:extension} for further discussion.

\section{Tax risk and stochastic control}\label{sec:mathematics}

In this section we  analyze the  case of tax risk, where the tax process $\btau$ follows a precise probabilistic model, namely a Markovian pure jump process with given jump intensity and jump size distribution. 
From now on we use the notation $\espt{\cdot}$ to indicate the conditional expectation given $X_t=x,  Y_t=y,\tau_t=\tau$. The reward function of the optimization problem \eqref{eq:optim-prob-producer} is thus given by 
\begin{equation}\label{eq:Unb_optimization}
  J(t,x,y,\tau,\bgamma) =  \espt{\int_0^T \!\!\! e^{-r(s-t)} \left( \Pi^*(X_s, \tau_s, Y_s)-\gamma_s-\kappa \gamma_s^2\right) \ud s +    e^{- r (T-t)} h(X_T) },
\end{equation}
and we denote by 
$V(t, x, y, \tau) = \sup\{J(t,x,y,\tau, \bgamma)\colon \bgamma \in \mathcal{A} \} 
$ 
the  corresponding value function.  The main goal of this section is to characterize   $V$   as viscosity solution of a certain Hamilton-Jacobi-Bellman (HJB) equation and to give criteria for the existence of a classical solution.

\subsection{The tax process}\label{sec:tax} 

We begin by introducing the dynamics of the tax process. Let $N(\ud t, \ud z)$ be a homogeneous Poisson random measure with intensity measure $m(\ud z)\ud t$, where $m(\ud z)$ is a finite measure on a compact set $\mathcal{Z}\subset \mathbb{R}$, i.e. $m(\mathcal{Z})=M<\infty$. Then we introduce the dynamics of the tax process as follows: 
\begin{equation}\label{eq:tau}
\tau_t = \tau_0 + \int_0^t\int_{\mathcal{Z}} \Gamma(t,Y_{t-},\tau_{t-},z) \ N(\ud t, \ud z), \quad t \in [0,T]
\end{equation}
for a function $\Gamma(t, y, \tau, z):[0,T]\times \mathbb{R}^d\times \mathbb{R}\times {\mathcal{Z}}\to \mathbb{R}$  such that equation \eqref{eq:tau} has a unique solution which is then automatically a Markov process. (A set of sufficient conditions is listed in Assumption \ref{ass:Lipschitz}-(iii)).  In the sequel, to avoid technicalities we assume that there exists $\tau^{\max}>0$ such that $\tau_t \in [0, \tau^{\max}]$ for all $t \in [0,T]$.  This translates into the following conditions on the function $\Gamma$: $\sup_{z \in {\mathcal{Z}}} \Gamma(t,y,\tau,z) \le \tau_{\max}-\tau$ and $\inf_{z \in {\mathcal{Z}}} \Gamma(t,y,\tau,z) \ge -\tau$, for all $t \in [0,T]$, $y \in \mathbb{R}^d$, $\tau \in [0,\tau^{\max}]$.

We denote by $\mathcal{L}^{\btau}$ the generator of $\btau$, for given $t$  and value of the factor process $y$, that is 
\begin{equation}\label{eq:generator_tau}
\mathcal{L}^{\btau}f(t,\tau,y) = \int_{\mathcal{Z}} \Big(f(t, y, \tau + \Gamma(t,y,\tau, z))- f(t,x,y)\Big) m(\ud z) 
\end{equation}
\begin{remark}
A simple example for the tax process $\btau$ is given by a two state Markov chain with states $\tau^1 < \tau^2$ and switching intensity matrix $G=(g_{ij})_{i,j\in\{1,2\}}$, where $g_{ii}=-g_{ij}$, $j\neq i$. For this example the generator of the tax process reduces to 
$\mathcal L^{\btau}f(\tau)=\sum_{j=1}^2 \I_{\{\tau=\tau^j\}} \sum_{i=1}^2 g_{ji} f(\tau^i)$. 
In this case, $\tau^1$ and $\tau^2$ represent low taxes and high taxes, respectively. A low tax regime ($\tau_t= \tau^1$) might correspond to a the decision of a government putting little emphasis on environmental policy, a high tax regime  to a government with a green policy agenda. This example will be discussed in detail in the numerical analysis. 
There are many ways to represent  a two state Markov chain as a pure jump process of the form \eqref{eq:tau}. For a specific construction let $\mathcal{Z} = \{0,1\}$,   $m^\tau(\ud z) = g_{12}\, \delta_{\{0\}}(\ud z)+g_{21}\, \delta_{\{1\}}(\ud z)$ and put 
$$
\Gamma(\tau,0) = [\tau^2-\tau]^+ - [\tau^1-\tau]^+ \text{ and } \Gamma(\tau,1) =  [\tau- \tau^2]^+ -  [\tau-\tau^1]^+ \,. 
$$
Note that the function $\Gamma$ is bounded and Lipschitz in $\tau$.   
\end{remark}

\subsection{Properties of the value function}\label{sec:preliminaries}

We need the following set of conditions for our analysis.  

\begin{ass}\label{ass:Lipschitz}
  \begin{itemize}
 \item[(i)] Function $h(x)$ is  Lipschitz in $x$.
 \item[(ii)] Functions $\beta, \alpha$ are continuous and globally Lipschitz. 
 \item[(iii)] Function  $\Gamma(t,y, \tau, z)$ 
  is  continuous in $t,y,\tau, z$, Lipschitz in $y, \tau$, for all $t \in [0,T]$, and for all $z \in {\mathcal{Z}}$ and satisfies $|\Gamma(t, y, \tau, z)|\le c (1+\|y\|)$. 
  \end{itemize}
\end{ass}

\begin{lemma} \label{lem:lipschitz}
Let $V$ be the value function of the problem \eqref{eq:b_optimization}.
Suppose that Assumption~\ref{ass:cost_regularity} and Assumption \ref{ass:Lipschitz} hold. Then 
\begin{itemize}
\item[(i)]$\Pi^*$ is Lipschitz continuous in $(x, y, \tau)$;
\item[(ii)] $V$ is Lipschitz in $x$, uniformly in $t, \tau, y$, with Lipschitz constant
\begin{equation}\label{eq:def_Lv}
L_V=\frac{L_{\Pi^*}(1-e^{-(r+\delta) T)})}{r+\delta}+  L_h;
\end{equation}
\item[(iii)] Suppose moreover that $\Pi^*$  and $h$ are increasing in $x$, then  $V $ is increasing in $x$ as-well.
\end{itemize}
\end{lemma}

\begin{proof}
We begin with Statement $(i)$. To prove this we will  use Assumption \ref{ass:cost_regularity}. By direct computations we get that
\begin{align}
&\left|\Pi^*(x^1, y^1, \tau^1) - \Pi^*(x^2,  y^2, \tau^2)\right|=\Big|\max_{q \in[0,q^{\max}]}\Pi(x^1, y^1,\tau^1, q)- \max_{q \in[0,q^{\max}]} \Pi(x^2, y^2,\tau^2, q)\Big |\\
&\qquad \leq \max_{q \in[0,q^{\max}]} \left|\Pi(x^1, y^1,\tau^1,  q)-\Pi(x^2, y^2,\tau^2, q)\right|\,.
\end{align}
Moreover, we get from the definition of $\Pi$ that 
\begin{align*}
   &\big | \Pi(x^1, y^1,\tau^1, q)  -\Pi(x^2, y^2,\tau^2, q) \big | \le  |p(y_1) - p(y_2)| q + |C_0(q,x^1,y^1) - C_0(q,x^2,y^2) |\\  & \qquad  + |\tau^1 C_1(q,x^1,y^1) - \tau^2 C_1 (q,x^2,y^2) | +  |\tau^1 -\tau^2| \nu_0(q)\\
  & \le q^{\max} L_p |y^1 -y^2| + L_{C_0}(|y^1 -y^2| +|x^1 -x^2|) + \tau^{\max}  L_{C_1}(|y^1 -y^2| +|x^1 -x^2|) \\ & \qquad + (||C_1||_\infty +||\nu_0||_\infty) |\tau^1-\tau^2|\,,
\end{align*}
so that $\Pi$ is Lipschitz continuous in $(x, y, \tau)$ uniformly in ${q \in[0,q^{\max}]}$. Here $||\cdot||_\infty$ represents the supremum norm. 

Next we establish $(ii)$. By $(i)$ we know that $\Pi^*$ is Lipschitz continuous in $(x, y, \tau)$ with Lipschitz constant $L_{\Pi^*}$. Next we let $X^1$ and $X^2$ be the solutions of equation \eqref{eq:stochastic_investment} with the initial conditions $x^1 \neq x^2$, respectively, that is
\begin{equation}
X^{i}_t= x^i+ \int_0^t \left(\gamma_s  - \delta X_s\right) \ud s + \sigma W_t,
\end{equation}
for $i=1,2$. Then we get that
$X^1_t-X^2_t= (x^1- x^2 ) e^{-\delta t}$ for $ t \ge 0$ and
\begin{align}
&|V(t, x^1,  y, \tau)-V(t, x^2, y, \tau)|\\
& \le \sup_{\bgamma \in \mathcal{A}}\espt{\int_t^T \big |\Pi^*(X^1_s,  Y_s, \tau_s)- \Pi^*(X^2_s, Y_s, \tau_s) \big | e^{-r(s-t)} \ud s +e^{-r(T-t)} \big| h(X^1_T)- h(X^2_T)\big |}\\
&\le \sup_{\bgamma \in \mathcal{A}} \espt{\int_t^T L_{\Pi^*} |X^1_s-X^2_s| e^{-r(s-t)} \ud s + e^{-(r)(T-t)} L_h|X^1_T-X^2_T|}\\
&\le \int_t^T L_{\Pi^*} |x^1-x^2| e^{-(r+\delta)(s-t)} \ud s + e^{-(r+\delta)(T-t)} L_h|x^1-x^2|  \\
&= |x^1-x^2| \left(\frac{L_{\Pi^*}(1-e^{-(r+\delta) (T-t)})}{r+\delta}+ e^{-(r+\delta) (T-t)} L_h\right).
\end{align}
This shows that $V^{\bar \gamma}$ is Lipschitz in $x$, uniformly in $t, \tau, y$ with uniform Lipschitz constant $L_V$.

For $(iii)$ note that  if $h$ and $\Pi^*$ are increasing in $x$, the reward function \eqref{eq:Unb_optimization} of the problem is increasing in $x$, which carries over to the value function $V$ by definition.
\end{proof}

\begin{remark}[Comments and extensions] \label{rem:comments}
 \begin{itemize}
 \item[1.] The argument in the proof of Lemma \ref{lem:lipschitz}-(i) can be used to get regularity of the function $\Pi^*$ even in the case where an investment may expand the maximum capacity, which makes sense for instance in the setup of the example on two technologies. 
 \item[2.] It is possible to show that if $\Pi^*$  and $h$ are concave in $x$, then $V $ is also concave in $x$.   This situation  arises, for instance, if $\Pi(t,x,y,\tau, q)$ is concave in $x$ and if $q^*$ is a fixed quantity. 
\end{itemize}
The mathematical details of these two extensions are discussed in Appendix \ref{appendix:extension}.
\end{remark}

\subsection{Viscosity solutions}
For mathematical reasons, we first assume that the set of admissible controls is bounded and we denote by $\mathcal{A}^{\bar \gamma} \subset \mathcal{A}$ the set of all adapted càdlàg processes $\bgamma$ with $0 \le \gamma_t \le \bar \gamma$ for all $t$. Let
\begin{align} 
V^{\bar \gamma}(t,x, y,\tau):=\sup_{\bgamma \in \mathcal{A}^{\bar \gamma}}\espt{\int_t^T \left( \Pi^*(X_s, Y_s, \tau_s)-\gamma_s-\kappa \gamma_s^2\right)e^{-r(s-t)} \ud s + e^{-r(T-t)} h(X_T)}\,.  \label{eq:b_optimization}
\end{align}
As a first step we show that  $V^{\bar \gamma}$ is the unique viscosity solution of the HJB equation
\begin{equation} \label{eq:HJB-bounded}
\begin{split}
 v_t (t, x, y, \tau)& +\Pi^*(x,y, \tau )  + \mathcal{L}^{\btau} v(t, x, \tau) + \mathcal{L}^Y v(t,x, y,\tau)  +  \frac{\sigma^2}{2}v_{xx}(t, x, y, \tau) \\
  & + \sup_{0\le \gamma\le \bar\gamma} \big \{v_x (t, x,y, \tau) (\gamma - \delta x) - (\gamma + \kappa \gamma^2) \big \} =- r v(t, x,y, \tau) 
\end{split}
\end{equation}
with the terminal condition $v(T,x, y, \tau)= h(x)$. 

\begin{proposition} \label{prop:viscosity}
 The function $V^{\bar \gamma}$ is Lipschitz in $(x, y)$ and H\"older in $t$ and  the unique viscosity solution of the equation \eqref{eq:HJB-bounded}. Moreover, a comparison principle holds for that equation. 
\end{proposition}
\begin{proof}
It is easy to check that for the problem~\eqref{eq:b_optimization} the hypotheses (2.1)-(2.5) of \citet{pham1998optimal} are satisfied. Then, the result follows from  \cite[Theorem 3.1]{pham1998optimal}.
Note  that in  \cite{pham1998optimal}  it is assumed that the controls take values in a compact set, so that the results of that paper apply only to the case where $\bgamma \in \mathcal{A}^{\bar \gamma}$. 
\end{proof}

Next we want to prove that $V^{\bar \gamma}$ is independent of $\bar \gamma$ for sufficiently large values of $\bar \gamma$.
For this we use  that, in view of Lemma \ref{lem:lipschitz}, the value function $V^{\bar \gamma}$ is Lipschitz in $x$ with Lipschitz constant $L_V$ as in equation \eqref{eq:def_Lv}; in particular, the Lipschitz constant may be taken independent of $\bar \gamma$.

\begin{proposition}\label{prop:large_gamma}
Consider constants $\bar \gamma^1 <  \bar \gamma^2$ such that  $\frac{(L_V-1)^+}{2k}<\bar \gamma^1$. Then $V^{\bar \gamma^1}=V^{\bar \gamma^2}$.
\end{proposition}

\begin{proof}
Denote for $i=1,2$ by $V^{\bar \gamma^i}(t, x,y, \tau)$  the value function of the optimization problem with strategies in $\mathcal{A}^{\bar{\gamma}^i}$. 
Since $\bar\gamma^1<\bar\gamma^2$,  it is immediate that $\mathcal{A}^{\bar{\gamma}^1}\subset \mathcal{A}^{\bar{\gamma}^2} $ and hence $V^{\bar \gamma^2}(t, x, y,\tau)\ge V^{\bar \gamma^1}(t, x, y, \tau)$. To establish the opposite inequality, i.e. $V^{\bar \gamma^1}(t, x,y, \tau) \ge V^{\bar \gamma^2}(t, x,y, \tau)$, we prove that $V^{\bar \gamma^1}$ is a viscosity supersolution of the  HJB equation~\eqref{eq:HJB-bounded} with $\bar \gamma = \bar\gamma ^2$.

Fix some point $(t_0,x_0, y_0, \tau_0)$. Since $V^{\bar \gamma^1}$ is a viscosity solution (and hence in particular a supersolution) of \eqref{eq:HJB-bounded} with  $\bar \gamma = \bar\gamma^1$,  for every smooth  function $\phi(t,x,y,\tau)$  such that
\begin{equation} \label{eq:cond_supersol}
\phi(t,x, y,\tau)\leq V^{\bar \gamma^1}(t, x, y, \tau) \text{ for all }(t,x, y, \tau) \text{ and } \phi(t_0,x_0, y_0, \tau_0)=V^{\bar \gamma^1}(t_0, x_0, y_0, \tau_0)
\end{equation}
it holds that
\begin{align}
-& \bigg ( \phi_t (t_0, x_0, y_0, \tau_0)  +\Pi^*(x_0, y_0, \tau_0)  + \mathcal{L}^{\btau} V^{\bar \gamma^1} (t, x, y, \tau) + \mathcal{L}^Y \phi(t,x,y,\tau)  +  \frac{\sigma^2}{2}\phi_{xx}(t, x,y, \tau) \nonumber \\
  &\quad  + \sup_{0\le \gamma\le \bar\gamma^1} \big \{\phi_x (t, x, y, \tau) (\gamma - \delta x) - (\gamma + \kappa \gamma^2)\big \} - r V^{\bar \gamma^1}(t, x, y, \tau) \bigg ) \ge 0  
  \label{eq:bounded_phi}
\end{align}
It follows from \eqref{eq:cond_supersol} that
$ \phi_x(t_0,x_0, y_0,\tau_0)\le L_{V}$
 with $\phi_x$ being the partial derivative of $\phi$ with respect to $x$. Now note that the supremum in \eqref{eq:bounded_phi}
is attained at $\gamma^*=\frac{(\phi_x-1)^+}{2k}\le \frac{(L_V-1)^+}{2k}<\bar \gamma^1$. Hence  we can replace $\bar \gamma^1$ with $\bar \gamma^2$ in \eqref{eq:bounded_phi} without changing the supremum. This implies that $V^{\bar \gamma^1}$ is a supersolution of \eqref{eq:HJB-bounded} with  $\bar \gamma = \bar\gamma^2$ and completes the proof.
\end{proof}

Summarizing, we have the following result.
\begin{theorem} \label{thm:viscosity}
The value function $V$ of the optimization problem \eqref{eq:Unb_optimization} is Lipschitz in $(x, y)$, H\"older in $t$ and the unique viscosity solution of the HJB equation \eqref{eq:HJB-bounded} for any fixed  $\bar \gamma >\frac{(L_V-1)^+}{2k}$.
\end{theorem}

\begin{proof}
In the sequel we show that $V(t, x, y, \tau)=V^{\bar \gamma}(t, x, y, \tau)$ for $\bar \gamma >\frac{(L_V-1)^+}{2k}$, hence $V$ inherits the regularity properties of $V^{\bar \gamma}$ from Proposition \ref{prop:viscosity}.
\noindent 

In view of Proposition \ref{prop:large_gamma},  it remains  to show that $V(t, x,y, \tau)=\lim_{m \to \infty} V^{m}(t, x, y, \tau)$ ($V^m $ is the solution of  \eqref{eq:HJB-bounded} with $\bar \gamma =m$.)
The inequality $V(t, x, y, \tau) \ge \lim_{m \to \infty} V^{m}(t, x, y, \tau)$
is clear, since $\mathcal{A}^{m}\subset \mathcal{A}$.
For the converse inequality, we observe that for all $\bgamma \in \mathcal A$ there is a sequence of strategies $\bgamma^m \in \mathcal A^m$ such that $\lim_m \sup_{0 \le t \le T} |\gamma^m_t-\gamma_t|=0$ $\P$-a.s. To show this it is sufficient to take $\gamma_t^m=\gamma_t\wedge m$.
Moreover, it is easily seen that the reward function is continuous  with respect to $\bgamma$ so that  we have the convergence 
$J(t, x, y, \tau,,\bgamma^m)\to J(t, x, y, \tau, \bgamma)$.  Now we choose $\varepsilon>0$ and a strategy $\bgamma^\varepsilon \in \mathcal A$ such that $J(t, x, y, \tau, \bgamma^\varepsilon)\ge V(t,x, y,\tau)-\varepsilon/2$. Let $\{\bgamma^{m,\varepsilon}\}_{m \in \mathbb N}$ such that 
$$\lim_{m \to \infty} \sup_{0 \le t \le T} |\gamma^{m,\varepsilon}_t-\gamma_t^\varepsilon|=  0 \; \text{ $\P$-a.s.},$$ 
hence  $J(t, x, y, \tau, \bgamma^{m, \varepsilon})\to J(t, x, y, \tau, \bgamma^\varepsilon)$. Then, there is $m^*(\varepsilon)\in \mathbb N$ such that for all $m > m^*(\varepsilon)$ it holds that 
$V^{m}(t,x, y, \tau)\ge J(t, x, y, \tau, \bgamma^{m})\ge V(t,x, y,\tau)-\varepsilon$. Since $\varepsilon$ is arbitrary we get the result.
\end{proof}

\subsection{Classical solution} \label{subsec:classical}
In this paragraph we discuss conditions ensuring that the value function is a classical solution of an HJB PIDE. We recall that it is sufficient to work under the assumption that $0\le \gamma \le \bar \gamma$, in virtue of Theorem~\ref{thm:viscosity}.
We consider a cumulative investment $X$ process, the tax process $\btau$ and the factor process $Y$ as in our general framework, i.e. they are described by the equations \eqref{eq:stochastic_investment}, \eqref{eq:tau} and \eqref{equation:Y}, respectively. 

\begin{theorem}\label{thm:class_sol}
Assume that $\sigma>0$ and that there is some $\bar \alpha >0$ such that for all $\xi \in \R^d$, $\xi^\top \mathfrak{S}(t,y)\xi >\bar \alpha ||\xi||^2$. Then the value function $V(t, x, y, \tau)$ is the unique classical solution of the HJB equation \eqref{eq:HJB-bounded} for $\bar \gamma> \frac{(L_V-1)^+}{2\kappa}$.
\end{theorem}

The proof of this result is given in Appendix~\ref{appendix:tech_results}.

\begin{corollary}
Under the assumptions of Theorem \ref{thm:class_sol}, the optimal strategy satisfies $\gamma^*_t=\gamma^*(t, X_t, Y_t, \tau_t)=\frac{(V_x(t, X_t, Y_t, \tau_t)-1)_+}{2\kappa}$ for every $t \in [0,T]$.
\end{corollary}

\begin{proof}
Since the function $V$ is a classical solution of the equation \eqref{eq:HJB-bounded}, we get that $\gamma^*(t, X_t, Y_t, \tau_t)$ is the optimal strategy by verifying first and second order conditions.
\end{proof}

In appendix \ref{app:viscosity} we discuss an example which shows  that the assumption $\sigma>0$, i.e. strict ellipticality of the generator of the controlled process $X$, plays a crucial role to obtain the classical-solution characterization discussed in this section. This example illustrates in particular that, when the assumption is not satisfied, the value function may be a strict viscosity solution of the HJB equation. 



\section{Tax policy uncertainty and stochastic differential games}\label{sec:game}

Climate policy variables such as emission  tax rates are the result of unpredictable political processes. Hence it is difficult  to come up with a `correct' probabilistic model for the evolution of future emission tax rates.  In this section we therefore study the optimal investment into abatement technology for an uncertainty-averse producer who considers a set $\mathcal T $ of plausible tax processes but who does not assume  a precise probabilistic model for the future tax evolution. Instead, she  determines her  optimal production and investment strategy as the equilibrium strategy of a stochastic differential game with malevolent opponent. The objective of the producer remains that of maximizing expected profits. On the other hand, the opponent chooses a tax process from  $\mathcal T $  to minimize the  profits of the producer. From the viewpoint of the producer the tax process chosen by the opponent thus constitutes a worst-case tax scenario.  

Stochastic differential games have been used before as a tool for modelling the decision making of uncertainty averse investors.  For instance, \citet{bib:avellaneda-paras-96} or \citet{bib:herrman-karbe-seifried-17} use stochastic differential games to deal with model risk in the context of pricing and  hedging positions in derivative securities. Important contributions on the mathematical theory  of stochastic differential games  include \citet{bib:friedman-72}, \citet{bib:fleming-souganidis-89} or, more recently, \citet{bib:possamai-touzi-zhang-20}.

\subsection{The differential game} 
We now describe  the game between the producer and her opponent in detail. As in the tax risk case, the producer chooses her production and her investment, whereas the opponent chooses the tax rate. 
The dynamics of the factor process $Y$ and of the stochastic investment $X$ are given by \eqref{equation:Y} and \eqref{eq:stochastic_investment}, respectively. We consider the  profit function $\pi(q,x,y,\tau)$ introduced in  \eqref{eq:pi}, where the cost function satisfies the Assumption \ref{ass:cost_regularity}.  
We specify tax rates as follows.   
We assume that the set $\mathcal{T}$ consists of all adapted  tax processes with  values in a band around some deterministic  tax plan  $\bar \tau \colon[0,T] \mapsto [0,\infty) $. The tax plan $\bar \tau$ can be interpreted as the producer's prediction of the future tax evolution or as the future carbon tax rate officially announced by the government at $t=0$. 
Given functions $\tau^{\min}, \tau^{\max} : [0,T] \to [0,+\infty)$ with $\tau^{\min}(t)\le \bar \tau(t)\le \tau^{\max}(t)$ for every $t \in [0,T]$,  we define  $\mathcal T$ as  the set of all adapted processes $\btau=\left(\tau_t\right)_{0 \le t\le T}$ such that $ \tau^{\min}(t)\le \tau_t \le \tau^{\max}(t)$ for all $t \in [0,T]$.   
Next, we denote by $\mathcal Q$ the set of all adapted production processes $\bq=\left(q_t\right)_{0 \le t\le T}$  taking values in $[0, q^{\max}]$, for some $q^{\max}>0$ that represents the maximum capacity of production.
Finally, recall that $\mathcal A$ denotes the set admissible investment strategies, i.e. the set of all adapted process $\bgamma=\left(\gamma_t\right)_{0 \le t \le T}$ with values in $[0, +\infty)$ and $\mathbb E\left[\int_0^T \gamma_t \ud t\right]<\infty$. 

Given a tax process $\btau \in \mathcal T$, the producer chooses the investment rate $\bgamma\in \mathcal A$  and the quantity $\bq \in \mathcal Q$ of energy to be produced in order to maximise her expected profits given  by 
$$ \tilde J(\btau, \bgamma, \bq) = \mathbb{E} \Big [\int_0^T\big( \Pi(q_s, X_s, Y_s, \tau_s) - \gamma_s -\kappa \gamma_s^2  \big )e^{-r(s-t)} \ud s +h(X_T) e^{-r(T-t)} \Big ] \,.$$
Note that now the choice of $\bq$ is a part of the game and cannot be done upfront (other than in the tax risk case). Given an investment strategy $\bgamma$ and a production process $\bq$, the opponent on the other hand chooses the tax process $\btau \in \mathcal T$ in order to minimize the expected profit of the producer. In this problem tax processes are penalized via the   function   
\begin{equation} \label{eq:penalty}
\btau \mapsto  \rho(\btau) = \mathbb{E} \Big [\int_0^T \nu_1 (\tau_t-\bar \tau(t))^2 \,\ud t \Big ]\,,  
\end{equation}
for a fixed constant $\nu_1 >0$,  that is the opponent wants to minimize $\tilde J(\btau,\bgamma, \bq) + \rho(\btau)$. The interpretation is as follows: from the viewpoint of the uncertainty averse producer the tax process $\btau$ chosen by the opponent constitutes a \emph{worst-case tax scenario}. The penalty $\rho(\btau)$ reflects the plausibility of different tax processes from the viewpoint of the producer. In particular, a process $\btau$  which  deviates substantially  from $\bar \tau$ is considered implausible by the producer and it is therefore penalized strongly by  the penalty function $\rho(\cdot)$. 
The penalization function $\rho(\btau)$ is independent of $\bq$ and $\bgamma$. Hence it can be added to the objective function of the producer without altering his decisions. We may therefore model the game between the producer and the opponent as a zero sum game with reward function
\begin{align}
J(t,x,y,\btau, \bgamma, \bq) = & \mathbb{E}_t \Big [\int_t^T\big( \Pi(q_s, X_s, Y_s, \tau_s) + \nu_1 (\tau_s-\bar \tau(s))^2  \\ & \quad - \gamma_s -\kappa \gamma_s^2  \big )e^{-r(s-t)} \ud s +h(X_T) e^{-r(T-t)} \Big ]\,.\label{eq:reward_game}
\end{align}
Following \citet{bib:friedman-72} we call a pair of strategies $(\bgamma^*,\bq^*)$ (for the producer) and  $\btau^*$  (for the opponent) an \emph{equilibrium} for the game if for any $\btau \in \mathcal{T}$, $\bgamma\in \mathcal A, \bq \in \mathcal Q$,
\begin{equation} \label{eq:equilibrium}
   J(0,X_0,Y_0,\btau^*,\bgamma,\bq) \le J(0,X_0,Y_0,\btau^*,\bgamma^*,\bq^*) \le  J(0,X_0,Y_0,\btau,\bgamma^*,\bq^*) \,.
\end{equation}
We then call $u(t, x , y):=J(t,x,y,\btau^*,\bgamma^*,\bq^*)$ the \emph{value} of the game. 
In the sequel we show that under certain regularity conditions  the  game \eqref{eq:reward_game} has  equilibrium strategies in feedback form, which implies that the value of the game is well defined. 

\subsubsection*{Comments.}  Note that in the game \eqref{eq:reward_game},  tax uncertainty is modelled by the width of the band around $\bar \btau$ and by the size of the constant $\nu_1$ in the penalty function \eqref{eq:penalty}, where a wider band or a smaller value of $\nu_1$ correspond to an increase in (perceived) uncertainty. 
Indeed, a large value of $\nu_1$ implies that  tax processes deviating strongly from $\bar \tau$ are strongly penalized and hence rarely  chosen by  the opponent, so that uncertainty is reduced. 

Finally, we caution  against an interpretation of the opponent in this game  as a regulator or the government. Indeed, a reasonable objective function for a government that wants to maximise the social welfare should account for relevant quantities such as overall emissions,  energy production or  tax revenue which are not part of  the reward function \eqref{eq:reward_game}. For an example of a `reasonable' reward function  of a regulator  we refer to \citet[Eqn. (18)]{bib:carmona-et-al-21} . 

\subsection{Characterisation of equilibrium strategies.}

In the sequel we aim to characterise the value of the game and the equilibrium strategies. In the context of stochastic differential games this is usually done via a  suitable  Bellman-Isaacs equation. However, since in our model  the tax value $\tau$ chosen by the opponent  affects only the running reward, the Bellman-Isaacs equation can be reduced to a standard HJB equation.     

We define the function  
\begin{equation}\label{eq:def-g}
    g (q, \tau; x,y) =  \Pi(q, x, y, \tau)  + \nu_1 (\tau-\bar \tau(t))^2  \,, 
\end{equation}
and recall that $\Pi(q, x, y, \tau)= p(y) q - [C_0(q, x, y) - C_1(q, x, y) \tau] + \nu_0(q) \tau $. 
In Lemma~\ref{lem:saddle-point}  we show that for every fixed $(x,y)$, the function  $g$ admits a unique saddle point $(q^*,\tau^*) $. Hence  we may define functions  $\hat q(x,y)$ and $\hat \tau(x,y)$ that map  $(x,y)$ to the associated saddle point of $g$. Denote by 
\begin{equation}\label{eq:def-G}
    G(x,y) = g (\hat q(x,y),\hat \tau(x,y),x,y) =  \max_{q} \min_{\tau}   g (q, \tau ; x,y) = \min_{\tau}   \max_{q} g (q,\tau; x,y)
\end{equation}
the corresponding saddle value, where the maximum is taken over  $q\in[0,q^{\max}] $ and the minimum over $\tau \in [\tau^{\min},\tau^{\max}]$. In the next result we show that the equilibrium strategy and the value of the game can be characterised in terms of an HJB equation with running reward given by the function $G$.

\begin{proposition}\label{prop:verification}
Suppose that for fixed $(x,y)$  the function $g$ has a saddle point  $(\hat q(x,y),  \hat \tau(x,y))$ and that the PDE 
 \begin{equation}\label{eq:Bellman-Isaacs-simplified}
u_t(t,x,y)  + G (x, y) + \mathcal{ L}^Y u (t,x,y)   + \frac{\sigma^2}{2} u_{xx}(t,x,y) + \sup_{\gamma \ge 0} \left(\gamma u_x(t,x,y) -\gamma - \kappa \gamma^2\right)  =  r u \,
\end{equation}
with the final condition $u(T, x, y)=h(x)$ has a classical solution.  Let $\hat \gamma(t,x,y) = (u_x(t, x,y) -1)^+/(2\kappa) $. Then $u$ is the value function of the game and   the strategies   $\bq^* = (\hat q(X_t,Y_t))_{0 \le t \le T} $, $\bgamma^*= (\hat \gamma(t,X_t,Y_t))_{0 \le t\le T}$ and $\btau^* = (\hat \tau(X_t,Y_t))_{0 \le t\le T}$  are equilibrium strategies for  the game. 
\end{proposition}

\begin{proof}
This proposition can be established via  classical verification arguments. Suppose that the opponent  uses the strategy $\btau^*$ and denote by $X$ the solution of the SDE 
$$\ud X_t = (\hat \gamma(t,X_t,Y_t) -\delta X_t )\ud t + \sigma \ud W_t \,.$$ Since  $(\hat q(x,y),  \hat \tau(x,y))$ is a saddle point of $g$, we have 
$G(x,y) = \sup_{q \in [0,q^{\max}]} g (q,\hat \tau(x,y);x, y)  $, and we
may rewrite the PDE~\eqref{eq:Bellman-Isaacs-simplified} in the form
\begin{align}
u_t(t,x,y)  & + \mathcal{ L}^Y u (t,x,y)   + \frac{\sigma^2}{2} u_{xx}(t,x,y) -\delta x u_x(t,x,y) \\
&+ \sup_{q \in [0,q^{\max}]} g (q,\hat \tau(x,y);x, y) 
  +  \sup_{\gamma \ge 0} \left\{\gamma u_x(t,x,y) -\gamma - \kappa \gamma^2\right\}  =  r u(t,x,y) \,. 
\end{align}
Moreover $u(T, x, y)=h(x)$, so that  this is the HJB equation for the  control problem 
\begin{equation} \label{eq:control-prob-producer}
    \max_{\bq\in \mathcal{Q}, \bgamma \in \mathcal{A} } \espt{ \int_t^T  \left (g (q_s,\hat \tau(X_s,Y_s);X_s, Y_s ) - \gamma_s- \kappa \gamma_s^2 \right ) e^{- r(s-t)} \ud s + e^{-r(T-t)} h(X_T) } \,.
\end{equation}
A standard verification result for stochastic control problems such as Theorem 3.5.2 in \citet{pham2009continuous} now shows that $u$ is the value function for the control problem \eqref{eq:control-prob-producer} and that $\bq^*$ and $\bgamma^*$ are an optimal strategy in \eqref{eq:control-prob-producer}. A similar argument shows that $\btau^*$ is optimal against $\bq^*$ and $\bgamma^*$, which completes the proof. 
\end{proof}

Next we verify that the assumptions of Proposition \ref{prop:verification} are satisfied. We begin with  the existence of a unique saddle point for $g$.   We omit the arguments $x,y$ to ease the notation.  
For fixed $q \in [0,q^{\max}]$ the function $\tau \mapsto g(q,\tau)$ is strictly convex and has a unique minimum on $[\tau^{\min} , \tau^{\max}]$ which we denote by $\tau(q)$. Similarly, the function $q\mapsto g(q,\tau)$ is strictly concave and has a unique maximum $q(\tau) $ on $[0,q^{\max}]$. A saddle point $(q^*,\tau^*)$ of $g$ is characterized by the equations
\begin{equation}\label{eq:char_saddle_point}
    \tau^*= \tau(q^*) \quad \text{ and } q^* = q(\tau^*).
\end{equation}
We use first order conditions to identify $\tau(q) $ and $q(\tau)$. It holds that 
\begin{equation} \label{ref:tau_q}
    \tau(q) = \Big \{\bar \tau  + \frac{1}{2 \nu_1} \big(C_1(q) - \nu_0 (q))   \Big \} \vee \tau^{\min} \wedge \tau^{\max}
\end{equation}
The optimal instantaneous production $q(\tau)$ is determined   as in Section~\ref{sec:optim-problem}. In particular, the  FOC characterizing $q(\tau)$ is 
$p - \partial_q C_0(q) - (\partial_q C_1(q)  - \partial_q \nu_0(q) )\tau =0$, and   $q(\tau)$ is therefore given by \eqref{eq:cases}. The existence of a unique solution to equation \eqref{eq:char_saddle_point} is established in the next lemma, whose proof is given in Appendix~\ref{app:game}. 

\begin{lemma}\label{lem:saddle-point} Suppose that the cost function $C$ satisfies Assumption~\ref{ass:cost_regularity} and that the functions  $C_0$, $C_1$ and  $\nu_0$ are moreover $C^2$ in $q$. Then,  for every fixed $(x,y)$,  the function $g(q,\tau;x,y)$ has a unique saddle point $(q^*,\tau^*) =: (\hat q(x,y),\hat \tau(x,y))$.    
\end{lemma}

The next theorem summarises the mathematical analysis of the stochastic differential game. 

\begin{theorem}\label{thm:Bellman-Isaacs}
    Suppose that $\sigma^2 >0$, that the generator $\mathcal{L}^Y$ is strictly elliptic, that the cost function satisfies  Assumption~\ref{ass:cost_regularity} and that $C_0$ and $C_1$ are moreover $C^2$ in $q$. Then the PDE~\eqref{eq:Bellman-Isaacs-simplified} has a unique classical solution $u$,  which is the value function of the game.  Moreover, the strategies   $\bq^* $, $\bgamma^*$  and $\btau^*$ from Proposition~\ref{prop:verification}  are equilibrium strategies for  the game. 
\end{theorem}

\begin{proof} In view of Proposition \ref{prop:verification}, we need to show the existence of a classical solution to the PDE~\eqref{eq:Bellman-Isaacs-simplified}.  For this we first  show that the function $G$ from~\eqref{eq:def-G} is Lipschitz in $(x,y)$. The definition of $G$ implies that 
    $|G(x',y') - G(x,y)| \le \sup_{(q,\tau) \in B} |g (q,\tau; x',y') -g (q,\tau; x,y)|  $, where $B = [0,q^{\max}] \times [\tau^{\min}, \tau^{\max}]$. 
    Now 
    \begin{align*}
    |g (q,\tau; x',y') -g (q,\tau; x,y)| &\le  q^{\max} |p(y')-p(y)| + |C_0(q,x',y') - C_0(q,x,y)| \\&\quad + \tau^{\max}| C_1(q,x',y') - C_1(q,x,y)| \\& \le C \vert (x',y') - (x,y) \vert\,,
    \end{align*}
    where the last inequality follows from the Lipschitz conditions in Assumption~\ref{ass:cost_regularity}. 
    Existence and uniqueness of a classical solution to \eqref{eq:Bellman-Isaacs-simplified} now follow by similar arguments as in Section~\ref{subsec:classical}. In fact, the analysis of \eqref{eq:Bellman-Isaacs-simplified} is even simpler than the analysis of the
    HJB equation in Section~\ref{subsec:classical}, since there are no jump terms in the equation. 
\end{proof}

\subsection{Properties of the optimal tax rate and production} \label{subsec:property-saddlepoint}

We continue with a few comments on the properties of the saddle point $(\tau^*,q^*) = \big( \hat\tau(x),\hat q(x)\big ) $, where we ignore the dependence on $y$ to ease the notation.  Note first that the rebate $\nu_0(q)$ plays an important role for  the form of the saddle point. 
From equation \eqref{ref:tau_q} we see that $\tau^* \ge \bar \tau$ if and only if $\nu_0(q^*)\le  C_1(q^*,x)$. 
In particular, without rebate, that is for $\nu_0\equiv 0$, we have   $\tau^* \ge \bar \tau$ so that  the anticipated tax rate  is higher than the reference tax value. Intuitively,  this incentivises the producer to invest more than she would do under the reference tax scenario, 
so that an increase in uncertainty is beneficial from a societal point of view, see Figure~\ref{fig:Iaverage_2tech_uncertainty}  for a  numerical confirmation and further discussion.  This is an interesting observation which distinguishes the stochastic differential game from the case where the model for the tax dynamics is known. 
On the other hand, with rebate and for full abatement, that is for $C_1 \equiv 0$, we have   $ \tau^* (q) < \bar\tau$. Hence   for full abatement the worst case tax value is lower then the reference tax. This is consistent with the objective of the opponent who  wants to minimize her payments to the producer.

\begin{figure}

\includegraphics[width=.48\textwidth]{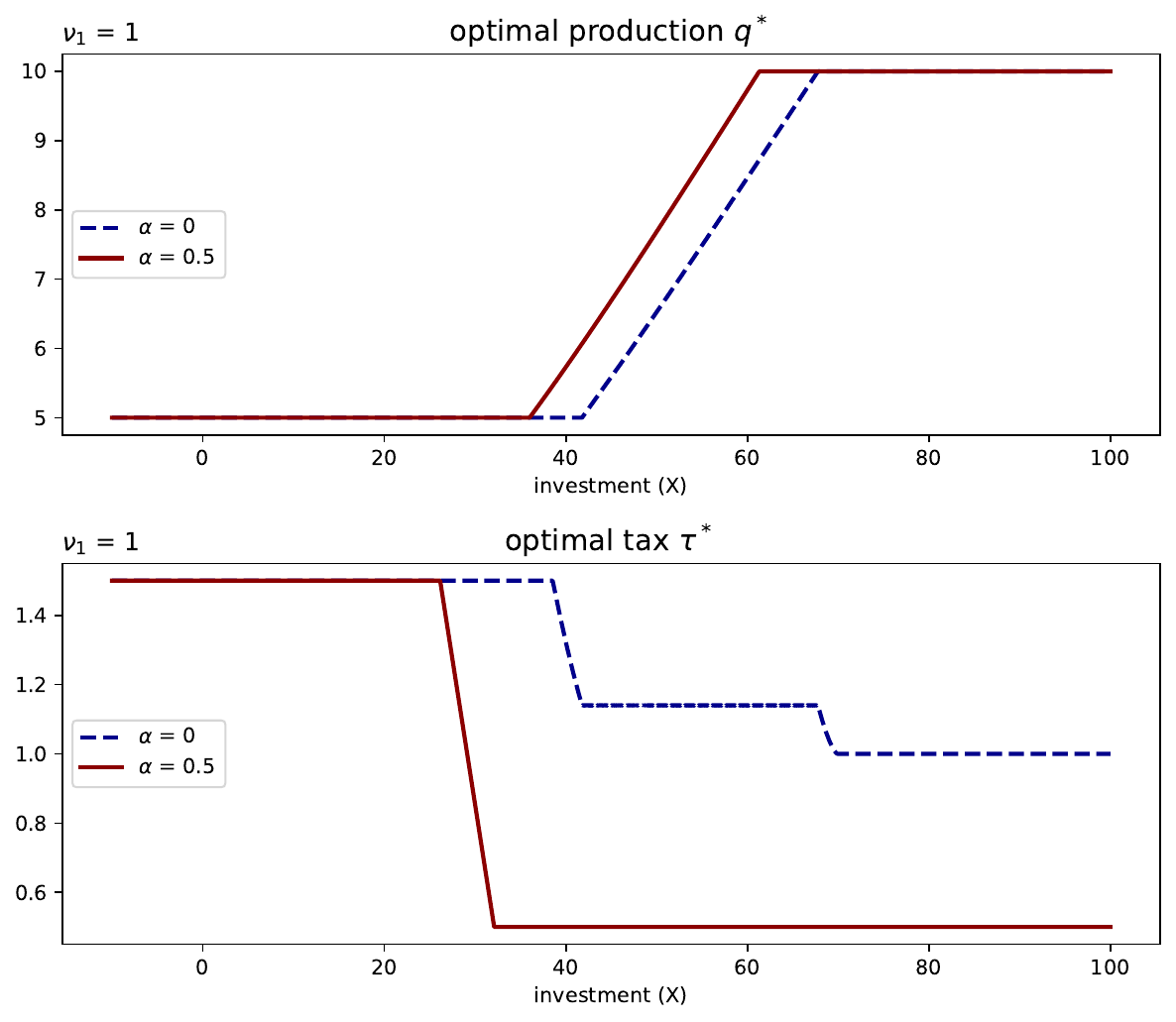} 
\includegraphics[width=.48\textwidth]{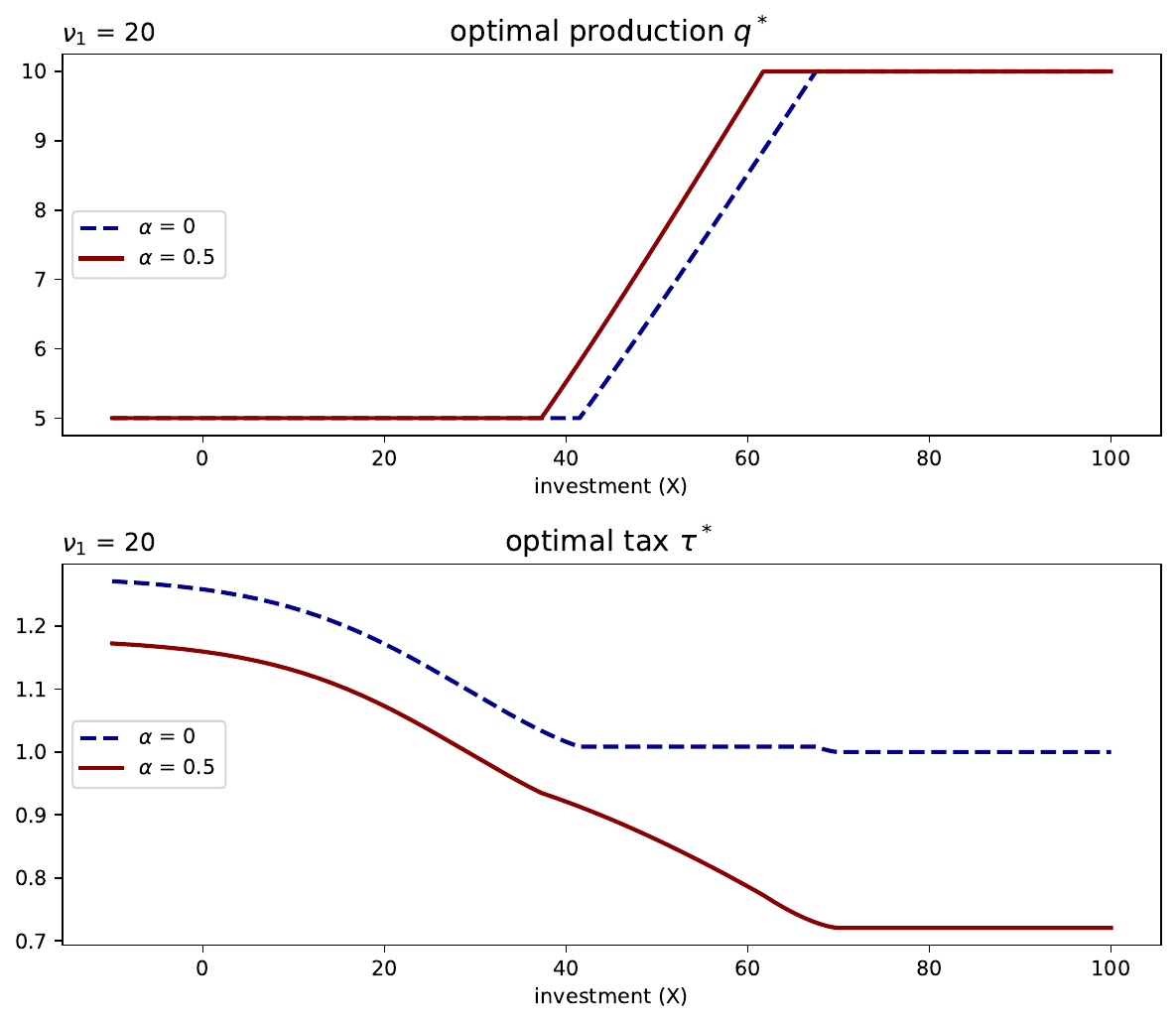} 

\caption{Representation of the saddle point $\big( \hat\tau(x),\hat q(x)\big )$ for the cost function from the two technologies example. The plots in the left panels  refer to the case of high uncertainty (small $\nu_1$), the plots in the right panels refer to  low uncertainty (large $\nu_1$). We take $\bar \tau =1$ in both cases. The value $\alpha=0$ corresponds to no rebate, $\alpha=0.5$ to rebate. For more details on parameters we refer to Section \ref{sec:experiments-two-tech}. Note that the left and the right  panel on the bottom use a  different scale.} \label{fig:saddle}
\end{figure}

These properties are illustrated in Figure~\ref{fig:saddle}.  
In this figure we provide representations of the saddle point $\big( \hat\tau(x),\hat q(x)\big )$ for a cost function corresponding to the example with  two production  technologies (see Section \ref{subsec:two_technology} for the model and Section \ref{sec:experiments-two-tech} for  parameter specifications). 
We fix a maximum capacity $q^{\max}=10$ and a minimum capacity $q^{\min}=5$. The lower bound on $q$ might correspond to contractual provisions stipulating a minimum amount of energy the producer has to provide at all times. In this example tax rates take value in the interval $[0.5, 1.5]$ and we fix the {most plausible} tax rate as  $\bar\tau=1$. We model the rebate  as $\nu_0(q) = e_b Q\left(\alpha q\right)$ for different values of $\alpha$, and we recall that the penalization for deviating from the most plausible tax rate is $\nu_1 (\tau-\bar \tau)^2$. 

The left panels corresponds to the case of high uncertainty, modeled by a small penalization  for deviating from  $\bar \tau$ ($\nu_1=1$), and those on the right correspond to the case of low uncertainty where deviations from $\bar \tau $ are strongly penalised ($\nu_1=20$). In all panels we consider the cases of no rebate $\alpha=0$ (blue dashed line) and rebate $\alpha=0.5$ (solid red line).
We see that  the optimal production $\hat q(x)$ is increasing in $x$. This is due to the fact that a higher investment level implies lower tax payments and hence a lower marginal cost.  Moreover, a  rebate boosts production, (the red solid line  is above the dashed blue line). Note finally, that in this example  optimal production is fairly robust with respect to the level of tax uncertainty as the function $\hat q$ is very similar in the left and in the right panel.  
The behaviour of tax rate $\hat \tau$ on the other hand is more sensitive to the choice of $\nu_1$.  In particular, for small $\nu_1$ (high uncertainty) the optimal tax rate assumes all values in the interval $[\tau^{\min}, \tau^{\max}]$ and the constraints $\tau\le \tau^{\min}=0.5$ and $\tau \ge \tau^{\max}=1.5$ are binding. In case of large $\nu_1$ (low uncertainty) on the other hand, these constraints are not binding and  the optimal tax value stays close to $\bar \tau$. This behaviour is consistent with formula \eqref{ref:tau_q}. Note finally that  $\hat \tau(\cdot) $ is decreasing in $x$. This is natural from an economic viewpoint, since for a high investment level emissions and hence the income from the  carbon tax are low, so that rebate and penalization lead to a lower value of $\tau^*$.


\section{Numerical experiments}\label{sec:numerics}

In this section we report the results of  numerical experiments that  study  the impact of transaction cost, production technology, market structure and randomness in the tax system on the investment strategy and the optimal electricity output  of the producer. In particular we identify certain situations where randomness in  taxes  reduces green investments, which is not desirable from a societal perspective. Throughout we use  the  deep-learning algorithm proposed in \citet{frey2022deep} to compute the value function and the optimal investment rate. We refer to Appendix \ref{appendix:numerical_methods} for the details  on the numerical methodology. 

In Section~\ref{sec:experiments-filter} we present results in the context of the filter technology from Section~\ref{subsec:filter_technology}, in Section~\ref{sec:experiments-two-tech} we discuss results for the two technologies from Section \ref{subsec:two_technology}. In both cases we work under tax risk and   assume that the tax process follows a Markov chain with two possible states $\tau^1 =0$ and $\tau^2 >0$ and transition intensity matrix $G$. This is a special pure jump process with fixed jump sizes that  allows us to capture typical features of a tax process with a small number of parameters.  
We study two special models for the tax evolution, namely the \emph{tax increase} and the \emph{tax reversal}.
In the tax increase case we assume that  $\tau_0 = 0$ and that  the process  jumps upward to $\tau^2>0$ at a  random time. This is a stylized model for the situation where  a government plans to rise carbon taxes in order to comply with    international climate agreements but where the exact timing of the tax rise depends on random political factors. In the  numerical experiments we moreover  assume that the high tax value is an absorbing state and we fix the transition intensities as $g_{12}=0.25$ and $g_{21}=0$. 
In the tax reversal case the tax is initially high ($\tau_0=\tau^2$) but jumps down to $\tau^1$ at a  random time. 
Such a downward jump might occur as a result of lobbying activities or  of a change in government composition. In our numerical experiments  we  fix the transition intensities as $g_{12}=g_{21}=0.25$. Note that  this choice implies  that taxes may jump up again at a later time point. 

In Section~\ref{sec:tax_uncertainty_numerics} we finally discuss examples for the stochastic differential game in the context of the two technologies.

\subsection{Experiments for the  filter technology under tax risk}\label{sec:experiments-filter}

We now discuss results of numerical experiments for the filter technology. We use   the following parameters:   $\delta=0.05$, $\sigma=0.05$, $r=0.02$, the time horizon is $T=15$ years and $h(x)=0$, which is in line with the fact that filters loose their value at the end of the lifetime of the underlying power plant.  We consider two possible parameters for  transaction costs, $\kappa=0.2$ or $\kappa=0.5$, which we refer  to as \emph{low} and \emph{high}  transaction costs, respectively. The high tax value is set to  $\tau^2 =0.2$. 
We work with a \emph{cost function} of the form \eqref{eq:cost_funct_2tech}, where the cost of one unit of raw material is constant and equal to $\bar c$, the quantity of raw material is specified as 
$Q(q)=aq^{\frac{3}{2}}$,  and where the abatement function is given by 
$$e(x)= 
\begin{cases} e_1 x + \frac{e_1^2}{4 e_0} x^2 & \mbox{if } x\leq e_0\,,\\
e_0 & \mbox{if } x>e_0\,.\end{cases}$$ 
In the numerical experiments we use the parameter values  $a=1.25$, $\bar c=1$, $e_0=1.5$, $e_1=0.5$. These parameter values were  chosen to obtain a  qualitatively reasonable behaviour of the production function, they were however not  calibrated to a real production technology.   
Note that for the chosen parameters the abatement cost is globally non-decreasing in $x$, concave and  differentiable and that the maximum abatement level is $e_0$.  

We consider two different market structures. In Section~\ref{sec:fixed_price} we study the case where the amount of electricity to be produced is fixed; in Section~\ref{sec:endogenous-production} we assume that  the electricity output  is endogenous and chosen by the profit-maximizing producer.

\subsubsection{Fixed electricity output}\label{sec:fixed_price}

In this section we assume that  electricity production is fixed and equal to $q^{\max} =4$, for instance since the producer has entered into long-term delivery contracts. In that case the investment decision of the electricity producer is independent of the rebate and of the form of the electricity price, so that we focus only on the randomness in the tax rate. 

In Figure~\ref{fig:single_traj} we plot single trajectories of the cumulative investment for the tax increase (left panel) and the tax reversal (right panel), for different values of the transaction cost parameter. In line with economic intuition, in both cases investments are larger for lower transaction costs. Moreover, the investment level decreases as time approaches the horizon date $T$. This is due to the fact that $\gamma_t^*$ is equal  to zero for $t$ close to $T$, since in that case  the tax savings  generated by new investment are too small to warrant the  expenditure. Finally, in both cases the producer reacts to changes in the tax regimes. Indeed, when a change in the tax rate occurs the trajectory of the investment process suddenly exhibits a change in the slope (i.e. a kink), which, intuitively, corresponds to a jump in the investment rate. In particular, in the tax increase scenario the investment rate $\gamma_t$ jumps upward  as  the tax rate switches  from $\tau^1$ to $\tau^2$. Interestingly, the producer starts to invest already at $t=0 $, even if the tax rate is equal to zero for small $t$. In this way he \emph{hedges} against an anticipated   tax increase. In fact, due to  transaction costs it would be too costly  to wait  until the upward jump in taxes actually occurs and to  invest only thereafter. This hedging behaviour distinguishes our model from the real options literature such as \cite{bib:fuss-et-al-08}, where it is optimal to wait if and when  a regulator acts and to invest only afterwards. In the tax reversal case, investments starts at a high rate due to the high taxation of emissions. As soon as taxes switch to $\tau^1$ the producer reduces or even stops her investment so that $X_t$ decreases due to  depreciation. 

\begin{figure}[ht]
  \centering
  \includegraphics[width=.48\textwidth]{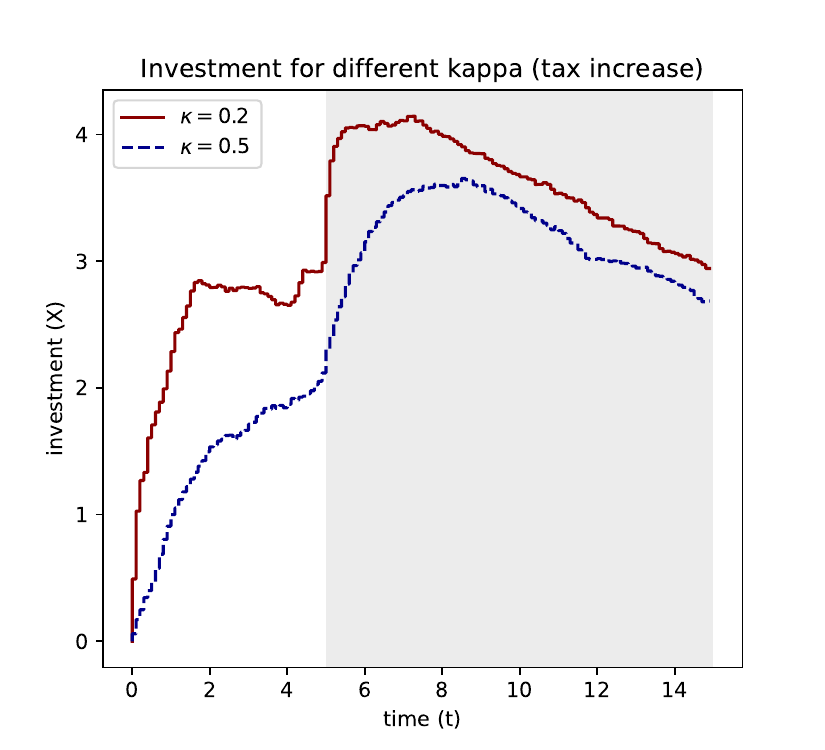}
    \includegraphics[width=.48\textwidth]{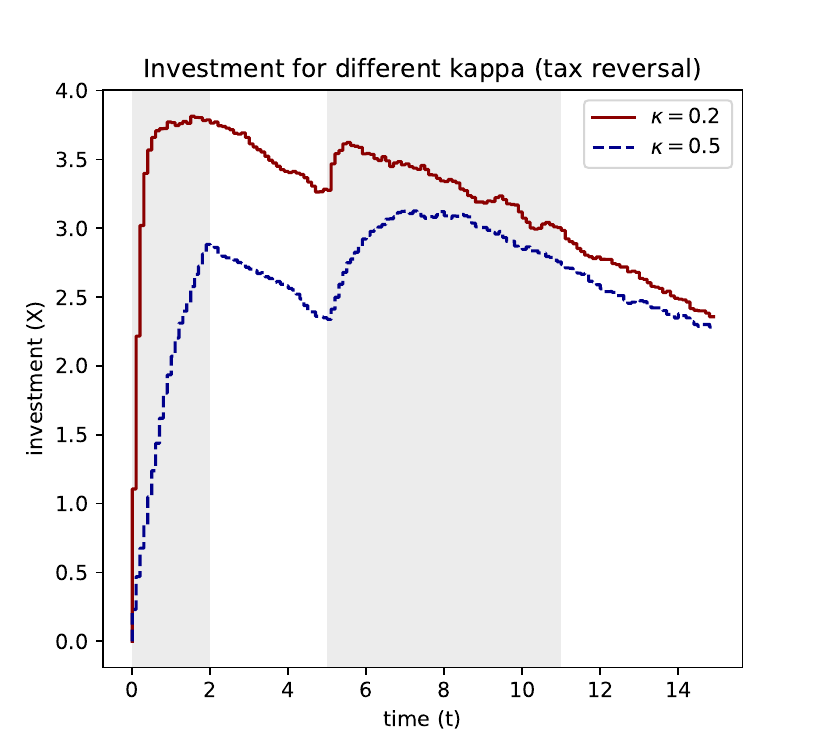} 
\caption{Single trajectory of cumulative  investment $X$ for the tax increase case (left panel) and the tax reversal case (right panel), under low transaction costs (solid line) and high transaction costs (dashed line). The grey and the white shaded areas correspond to time periods with high tax rate and low tax rate, respectively.} \label{fig:single_traj}
\end{figure}

In Figure \ref{fig:average_inv} we plot the evolution over time of the average investments $\mathbb{E}\left[ X_t \right]$ together with the 5\% and 95\% quantiles of the distribution of $X_t$, for every $t \in [0,T]$. For comparison purposes we moreover plot the optimal investment in a deterministic \emph{benchmark scenario} $\bar \tau (t)$, which is computed as follows: in the   tax increase case we assume that $\bar\tau (t)$ is linear increasing that is $\bar \tau (t) = bt$; in the tax reversal case we assume that the reverence tax rate is constant, $\bar \tau (t) =\bar\tau$. In both cases we assume that the expected average tax rate is identical in the benchmark scenario and in the case with random taxes. For the tax increase  case we therefore have the condition 
\begin{equation}\label{eq:benchmark-increasing}
\mathbb{E}\left[\int_0^T \tau_t \ud t \right] = b \frac{T^2}{2}, \text{ that is } b = \frac{2}{T^2} \mathbb{E}\left[\int_0^T \tau_t \ud t \right]  \,,
\end{equation}
which leads to $b= 0.0197$. For the tax reversal scenario we have 
\begin{equation}\label{eq:benchmark-constant} 
\bar \tau = \frac{1}{T} \mathbb{E}\left[\int_0^T \tau_t \,\ud t \right] \,,
\end{equation}
which leads to $ \bar\tau = 0.113.$ 

\begin{figure}[ht]
  \centering
  \includegraphics[width=.48\textwidth]{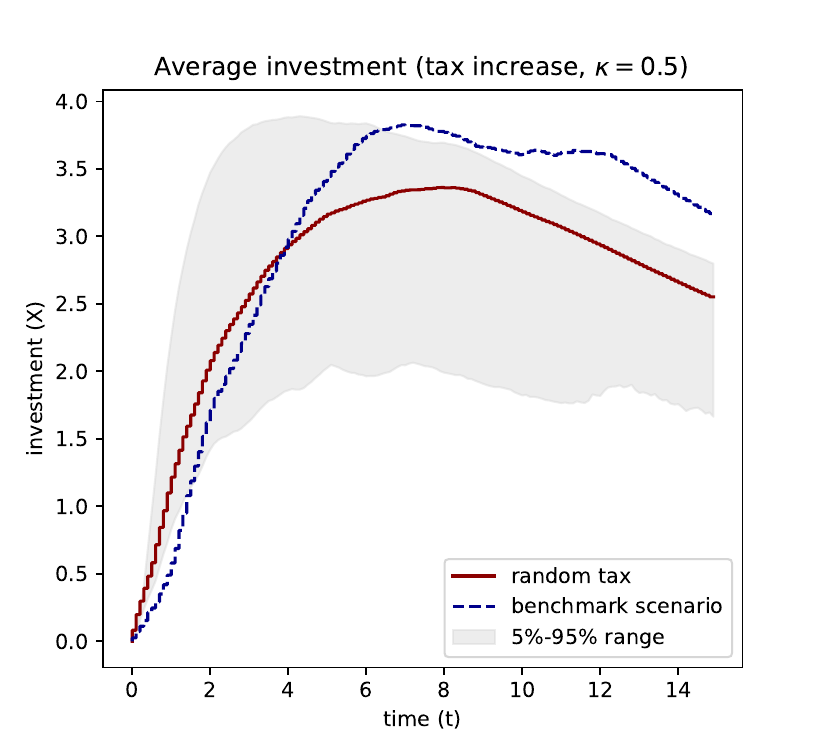}
  \includegraphics[width=.48\textwidth]{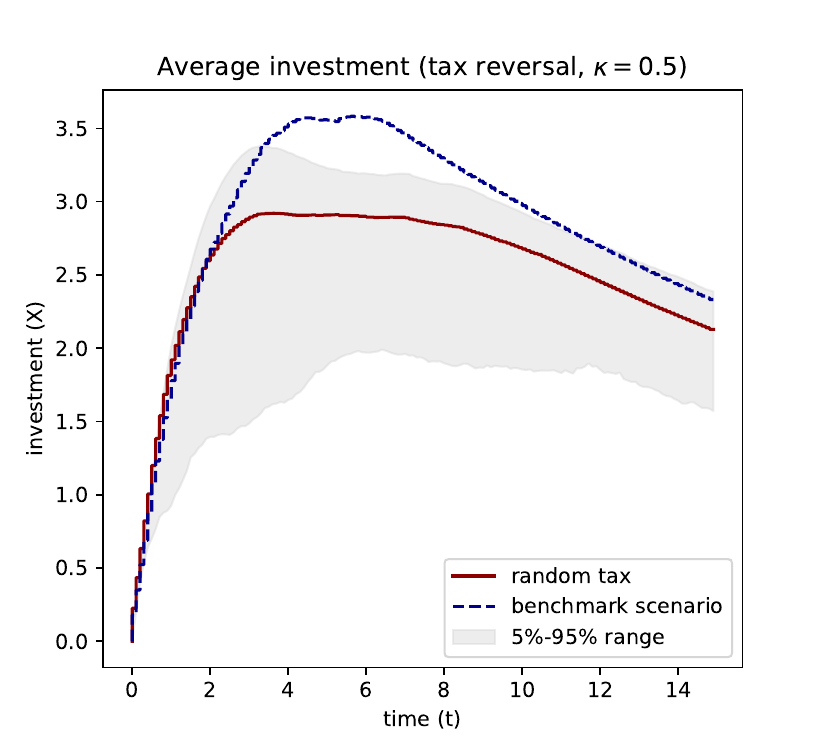}
\caption{Average total investment (solid red line) for the random tax increase  (left panel) and the random tax reversal (right panel), under high transaction costs ($\kappa=0.5$), versus total investment in the deterministic tax case (dashed blue line). The grey shaded areas correspond to the interval between the 5\% and the 95\% quantile of the investment level in the case with random taxes.  
} \label{fig:average_inv}
\end{figure}

Next we report the values for the average emissions at two evaluation dates, namely after 10 and 15 years. Table~\ref{tab:emissions_increase} contains the values for the random  tax increase, where the benchmark is the deterministic increasing tax rate, see \eqref{eq:benchmark-increasing}; Table~\ref{tab:emissions_reversal} gives  the values for the random tax reversal, where the benchmark is the constant tax rate, see \eqref{eq:benchmark-constant}. In the first three columns we report  the $5\%$ quantile, the mean and the $95\%$ quantile of the emission distribution in case of random taxes, in the fourth column we report  the level of emission for the benchmark case, for two different transaction costs parameters. The values in these tables suggest that the benchmark tax rate always leads to  emission levels  that are lower than the mean  emissions under random tax rates,  (in most cases emissions in the benchmark case are even below the $5\%$ quantile  of the emissions distribution).  This confirms the intuition that randomness in future tax rates reduces investments into carbon abatement technologies, so that   a deterministic tax policy would be beneficial for stipulating  emission reduction.

\begin{table}
    \centering 
    \begin{tabular}{|l|cccc|cccc|}\hline
    &\multicolumn{4}{c|}{$t=10$}&\multicolumn{4}{c|}{$t=15$} \\ 
                        & 5\% & mean  &95\%  & benchmark &5\% & mean  &95\%  & benchmark   \\ \hline
        $\kappa=0.2$    & 1.96 & 3.16 & 4.86  & 2.94  & 3.56   & 4.83    & 7.08     & 4.08   \\
        $\kappa=0.5$    & 3.33   & 5.27   & 7.82   & 5.07   & 5.21    & 7.31     & 11.34  & 6.38 \\ \hline
    \end{tabular}
    \vspace{.3cm}
     \caption{Quantiles of the emissions distribution for the random tax increase after $t=10$ and $t=15$ years. We assume that the quantity $q$ is fixed and equal to $q^{\max}=4$. 
     The benchmark tax leads  to lower emissions on average. 
    }
    \label{tab:emissions_increase}

    \centering 
    \begin{tabular}{|l|cccc|cccc|}\hline
    &\multicolumn{4}{c|}{$t=10$}&\multicolumn{4}{c|}{$t=15$} \\ 
                        & 5\% & mean  &95\%  & benchmark &5\% & mean  &95\%  & benchmark   \\ \hline
        $\kappa=0.2$    & 2.73  & 3.02   & 3.80   & 2.28    & 4.82    &  5.28    & 6.52  & 4.17 \\
        $\kappa=0.5$    & 4.31  & 5.11  & 7.34   & 4.18    & 6.74     & 7.83      & 10.78   & 6.52 \\ \hline
    \end{tabular}
    \vspace{.3cm}
     \caption{Quantiles of the emissions distribution for the random tax reversal after $t=10$ and $t=15$ years. We assume that the quantity $q$ is fixed and equal to $q^{\max}=4$. The constant tax leads  to lower emissions on average. }
   \label{tab:emissions_reversal}
\end{table}

Finally, we study how the credibility of an announced carbon tax policy affects the investment decision of the producer and hence the effectiveness of the policy. We begin with the case of the 
random tax increase. In Figure~\ref{fig:single_traj_beliefs}  we compare a path of the  cumulative investment of a producer  who does not belief in an announced future tax increase and who therefore works with a very low intensity ($g_{12} =0.05$)  to the investment path of an investor with $g_{12} = 0.25$,  both for $\kappa=0.5$ and for the same realization of the tax process.   We see that the investment of the investor with $g_{12} = 0.05$ is substantially lower, even if the tax path is the same.     
This is due to the fact that an investor who does not believe in a future tax increase does  not hedge against a future tax rise (see the discussion of Figure~\ref{fig:single_traj}) but he invests only \emph{after} the tax increase has actually materialized. 
The right panel of Figure~\ref{fig:single_traj_beliefs} corresponds to the tax reversal. We compare the optimal cumulative investment of an investor with $g_{21} = 0.25 $ and an investor who believes tax reversal  is very likely, that is  $g_{21} = 0.5$, for the same trajectory  of the tax process (we take  $g_{12} =0.25$ for both investors). We see that  the producer with $g_{21} = 0.5$ invests less then the investor with  $g_{21} = 0.25 $, even if both face the same tax trajectory. 
Table~\ref{tab:emissions_increase_wb} and Table~\ref{tab:emissions_reverse_wb} below, provide the quantiles and the average emissions for the tax increase case and the tax reversal case, respectively, for a producer with a wrong belief on the tax switching intensity versus a producer with the correct belief. These numbers show  substantially larger emissions in the wrong belief case, which confirm the behaviour depicted in Figure \ref{fig:single_traj_beliefs}. 

These experiments underline that a  carbon tax policy that is not credible (i.e.~producers  are  not convinced that an announced tax increase will actually be implemented or they expect  that a high tax regime will soon be reversed) is substantially less effective than a credible policy.

\begin{figure}[ht]
  \centering
  \includegraphics[width=.48\textwidth]{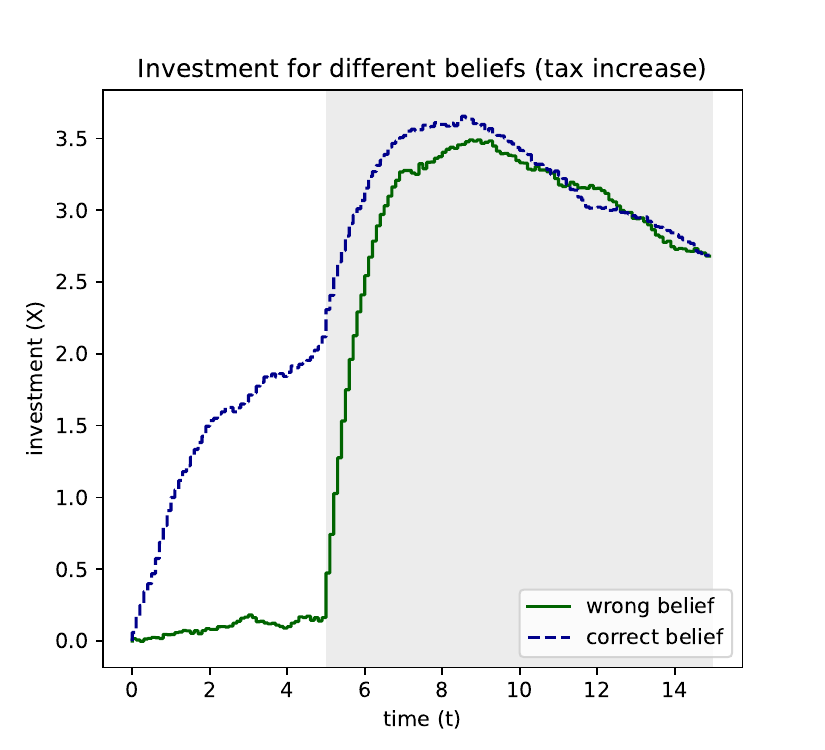}
  \includegraphics[width=.48\textwidth]{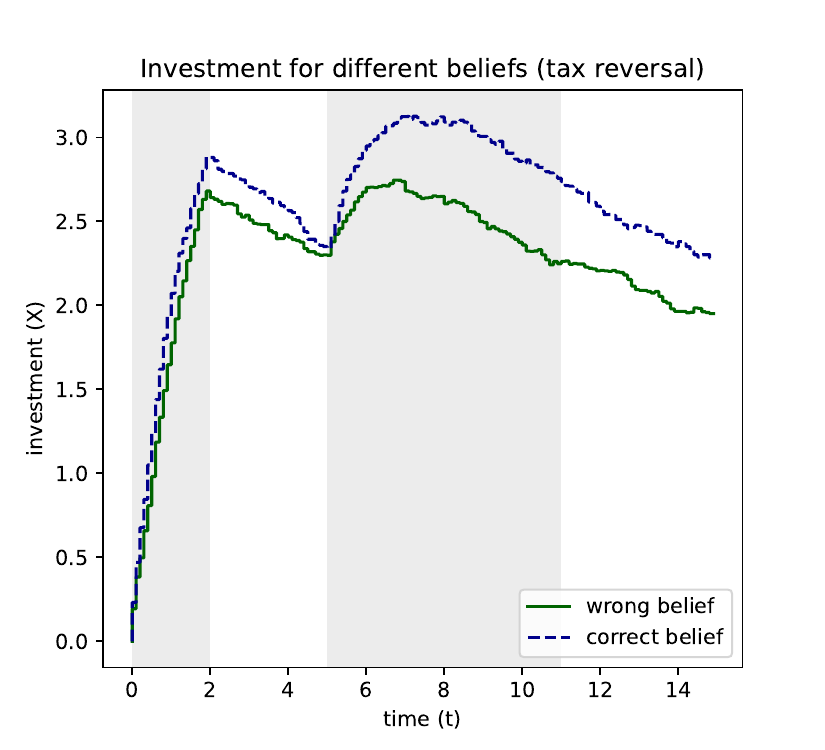}
\caption{Single trajectory of total investment for the tax increase case (left panel) and the tax reversal case (right panel), under high  transaction costs. Here we compare investors with different beliefs about switching intensity. In the tax increase case we compare an investor with $g_{12} = 0.25$ and an investor who believes tax increase is not very likely, that is   $g_{12} =0.05$ ($g_{21} =0$  for both investors) and we plot the investment for the same tax trajectory as in Figure \ref{fig:single_traj}  (left panel). In the tax reversal case  we compare an investor with $g_{21} = 0.25 $ and an investor who believes tax reversal  is very likely ($g_{21} = 0.5$) ($g_{12} =0.25$ for both investors) and we plot the investment for the same tax trajectory as in Figure \ref{fig:single_traj} (right panel).
\label{fig:single_traj_beliefs} }
\end{figure}

\begin{table}
    \centering 
    \begin{tabular}{|l|ccc|ccc|}\hline
    &\multicolumn{3}{c|}{$t=10$}&\multicolumn{3}{c|}{$t=15$} \\ 
                        & 5\% & mean  &95\%   &5\% & mean  &95\%     \\ \hline
        wrong belief    &    3.72&   7.96&    15.03&    5.54&     10.20&       21.41\\
        correct belief    & 3.33   & 5.27   & 7.82     & 5.21    & 7.31     & 11.34   \\ \hline
    \end{tabular}
    \vspace{.3cm}
     \caption{ \label{tab:emissions_increase_wb}Quantiles of the emissions distribution for the random tax increase after $t=10$ and $t=15$ years for an investor with a wrong belief versus a correct   belief on the switching intensity. We assume that the quantity $q$ is fixed and equal to $q^{\max}=4$. Wrong beliefs leads to substantially higher  emissions. 
    }
    
    \centering 
    \begin{tabular}{|l|ccc|ccc|}\hline
    &\multicolumn{3}{c|}{$t=10$}&\multicolumn{3}{c|}{$t=15$} \\ 
                        & 5\% & mean  &95\%   &5\% & mean  &95\%     \\ \hline
         wrong belief    &    5.13&   5.93&    8.67&    8.05&     9.23&       12.86\\
        correct belief     & 4.31  & 5.11  & 7.34     & 6.74     & 7.83      & 10.78    \\ \hline
    \end{tabular}
    \vspace{.3cm}
     \caption{\label{tab:emissions_reversal_wb}Quantiles of the emissions distribution for the random tax reversal after $t=10$ and $t=15$ years for an investor with a wrong belief versus a correct belief on the switching intensity. We assume that the quantity $q$ is fixed and equal to $q^{\max}=4$. Wrong beliefs leads to substantially higher  emissions. 
    }
   
\end{table}

\subsubsection{Stochastic price and endogenous electricity output}\label{sec:endogenous-production}

Now we consider a richer setup  where the selling price of electricity is random and where the producer optimizes the instantaneous electricity production $q^*_t=q^*(X_t, Y_t, \tau_t)$. 
We assume that the energy price is given by  $p_t=\exp(Y_t)$, where the process $Y$ is the solution of the one-dimensional SDE
\[
\ud Y_t= \theta(\mu - Y_t) \ \ud t + \alpha  \ \ud B_t, \quad Y_0=\ln(p_0), 
\]
for a one dimensional Brownian motion  $B=(B_t)_{t \ge 0}$ that is  independent of $W$. We fix $\mu=\ln(5)$, $\theta=1$, $\alpha=0.1$ and $p_0=5$. 
The dynamics of $X$ and $\tau$ are  as in Section \ref{sec:fixed_price}. 
In this framework we also consider a tax rebate which is modeled by the function $\nu_0(q)=\frac{1}{2} Q(q)e_0$, that is  the tax payments of the  producer are  fully refunded when half of the emissions are abated. 
The instantaneous profit is given by 
$$\Pi(q, x, y, \tau)=p(y) q - \left(Q(q)( \bar c + \tau (e_0 - e_1 x + \big( \frac{e_1^2}{4 e_0}\big ) x^2)^+)\right) + Q(q) \frac{e_0}{2} \tau.$$ 
Since in this example $Q(q) = a q^{\frac{3}{2}}$, we get  
\begin{equation}\label{eq:with_rebate}
q^*(x, y, \tau)=\left(\frac{2p(y)}{3a\left(\bar c+ \tau (e_0 - e_1 x + \big( \frac{e_1^2}{4 e_0}\big ) x^2)^+ - 1/2 \tau e_0 \right)}\right)^2 \wedge q^{\max}. 
\end{equation}
Note that in case there is no rebate, that is when taking $\nu_0(q)=0$, we have that 
\begin{equation}\label{eq:without_rebate}
q^*(x,y, \tau)=\left(\frac{2p(y)}{3a \left(\bar c+ \tau (e_0 - e_1 x + \big( \frac{e_1^2}{4 e_0}\big ) x^2)^+\right)}\right)^2 \wedge q^{\max}. 
\end{equation}

Figure \ref{fig:optimal_q} plots trajectories of the optimal production for  the random tax increase (left panel) and for the random tax reversal (right panel). In this example we set  the transaction costs parameter to $\kappa=0.5$. We compare the cases with rebate (solid black lines) with that of no-rebate  ($\nu_0(q)\equiv0$, solid grey lines). The plots are obtained for the same selected price trajectory.  
In these experiments we see that optimal production $q^*$ reacts to three different factors: (i) there are instantaneous jumps occurring at tax switches; (ii) between two consecutive jumps of the tax rate  production fluctuates as it adapts to changes in the  price; (iii) finally, the reaction of $q^*$ to tax switches  depends on the rebate. In particular, when a rebate is applied production is both larger and  more volatile than for $\nu_0(q)\equiv0$. This is in line with formulas \eqref{eq:with_rebate} and \eqref{eq:without_rebate}.  

\begin{figure}[ht]
  \centering
  \includegraphics[width=.48\textwidth]{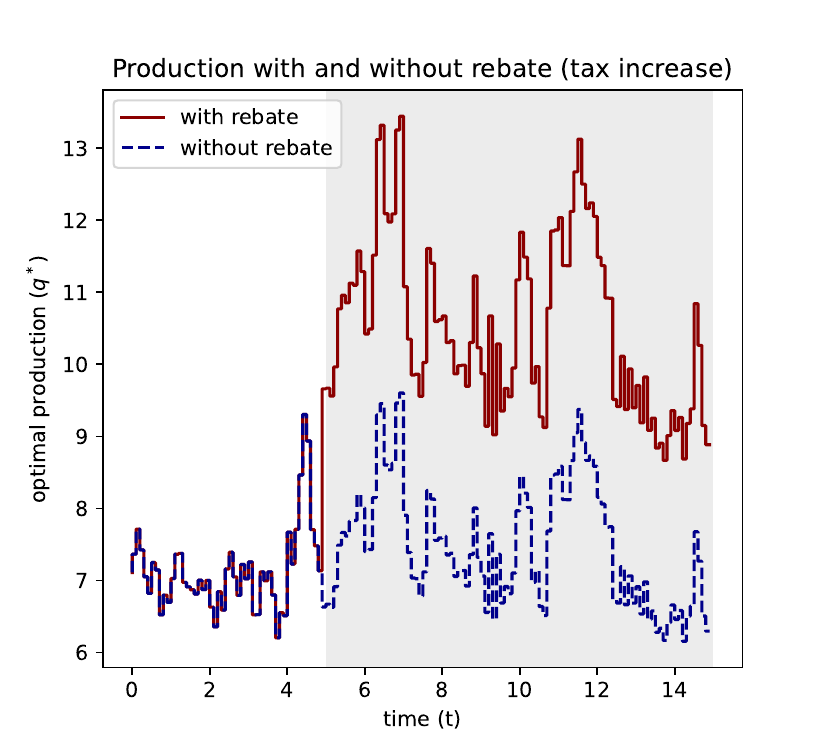}
  \includegraphics[width=.48\textwidth]{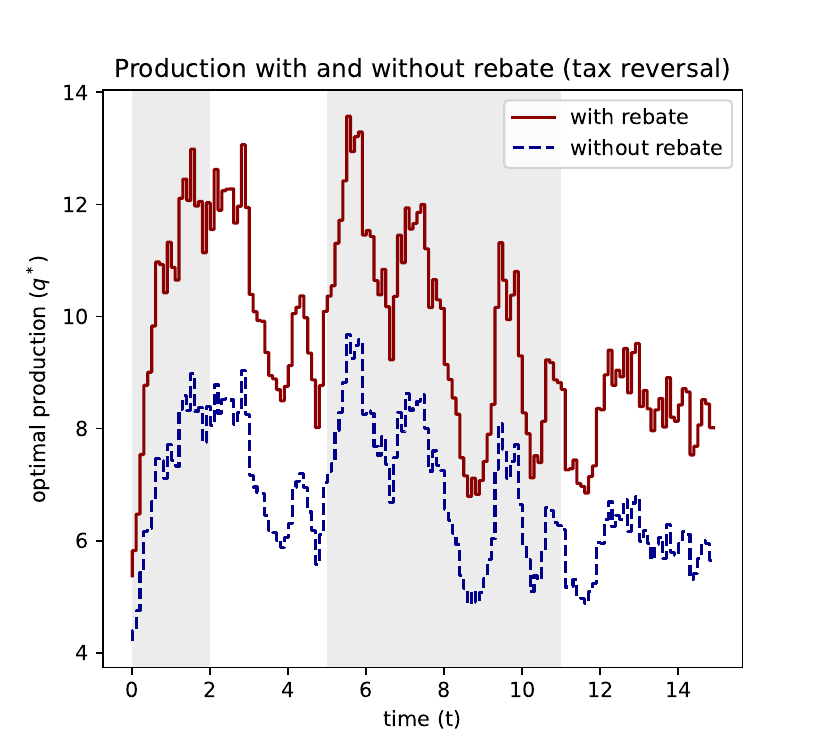}
\caption{Trajectory of the optimal production $q^*$ for the case with rebate (solid red line) and without rebate (dashed blue line). The left panel corresponds to the tax increase scenario and the right panel to the  tax reversal case. Both panels are obtained under the same selected price trajectory under transaction costs $\kappa=0.5$. The grey and white shaded areas correspond to high tax rate and low tax rate, respectively} \label{fig:optimal_q}
\end{figure}

The implications for the optimal  investment are depicted in Figure~\ref{fig:optimal_I}, where we consider the cumulative investment  for the same tax trajectories as in Figure~\ref{fig:optimal_q}.
We clearly see that the producer reacts to changes in the tax regime,  that the hedging effect (i.e. nonzero investment under zero tax level in anticipation of a tax switch) is still present for the random tax increase (left panel) and that hedging is  more prominent when a rebate is applied. We also see that investment levels decrease as $t$ approaches $T$, however this effect is less pronounced than in the case of  fixed electricity output displayed in Figure~\ref{fig:single_traj}.

Most importantly, these plots suggest that a rebate is in general beneficial for investment. To test our last observation we have also looked at the quantiles of the investment distribution with and without rebate. Precisely, we have computed average investments $\mathbb{E}[X_t]$, the $5\%$ quantile and the $95\%$ quantiles of its distribution after 10 and 15 years, and the investment values for benchmark cases. For a better  illustration of the results we have collected these numbers in Table~\ref{tab:investment_increase} for the tax increase and in Table~\ref{tab:investment_reversal} for the tax reversal. In both tables we compare the case with rebate and without rebate (on the first and second row of each table, respectively). The benchmark for the tax increase is computed from~\eqref{eq:benchmark-increasing}, the benchmark for the tax reversal is computed from~\eqref{eq:benchmark-constant}. 
The table supports our previous findings that rebate may be an important driver for investment. Indeed,  investments under rebate are always larger compared to the case where rebate is not applied. Moreover, we again find that, on average, randomness in future tax rates discourages investment. Investments in the benchmark cases, in fact, are always larger than the average investment under random taxes and most of the times above the $95\%$ of the distribution.

\begin{figure}[ht]
  \centering
  \includegraphics[width=.48\textwidth]{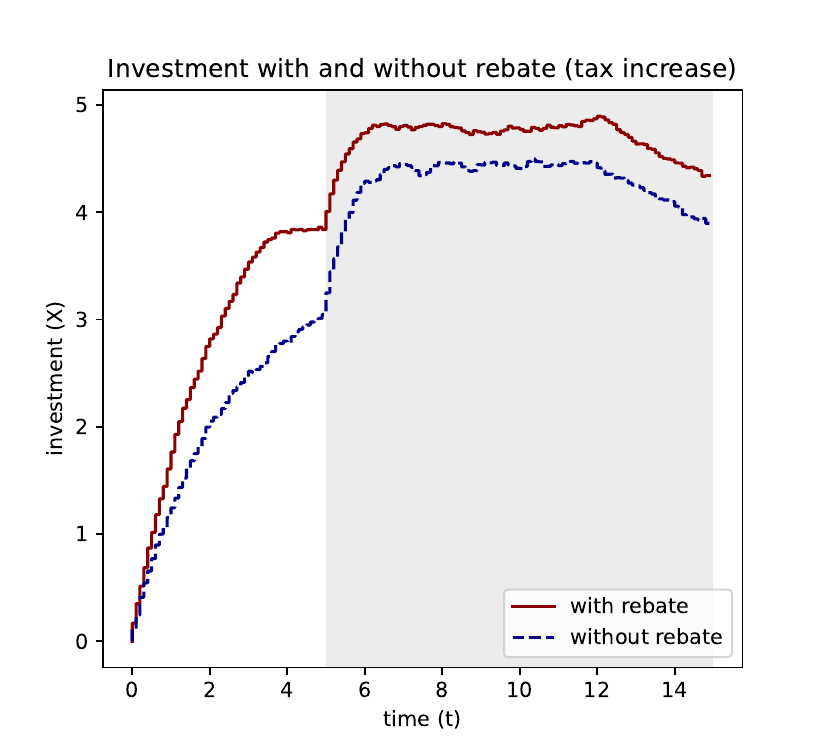}
  \includegraphics[width=.48\textwidth]{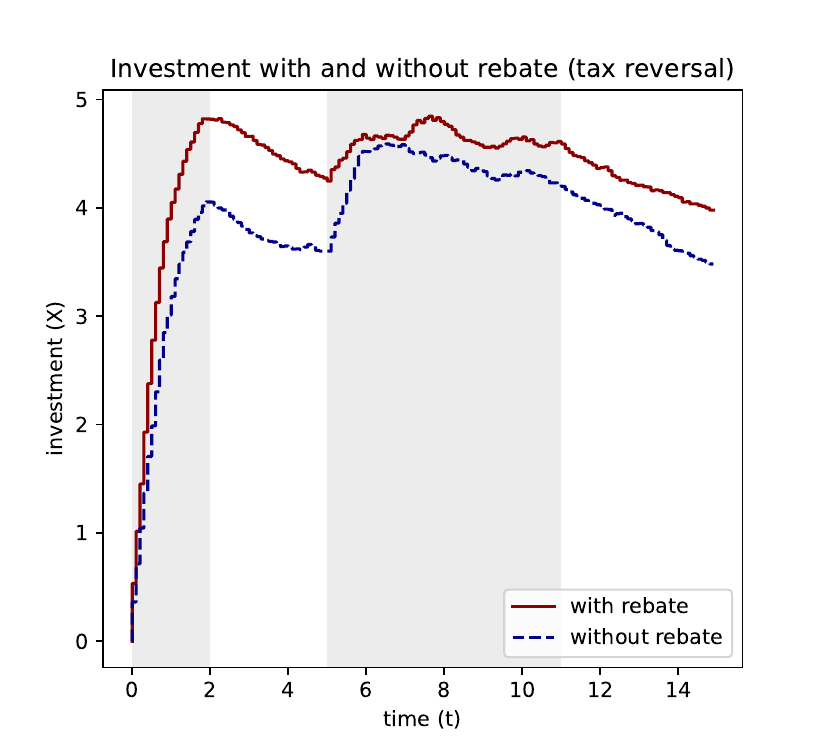}
\caption{Trajectory of the optimal investment $X$ for the case with rebate (solid red line) and without rebate (dashed blue line). The left panel corresponds to the tax increase scenario and the right panel to the  tax reversal case. Both panels are obtained under the same selected price trajectory under transaction costs $\kappa=0.5$. The grey and white shaded areas correspond to high tax rate and low tax rate, respectively
} \label{fig:optimal_I}
\end{figure}

\begin{table}[ht]
    \centering 
    \begin{tabular}{|l|cccc|cccc|}\hline
    &\multicolumn{4}{c|}{$t=10$}&\multicolumn{4}{c|}{$t=15$} \\ 
                        & 5\% & mean  &95\%  & benchmark &5\% & mean  &95\%  & benchmark   \\ \hline
        rebate    & 4.70  & 4.76    & 4.82  & 5.19   &  4.13   &  4.28    & 4.43 &  4.84 \\
        no-rebate  & 4.32 & 4.40   & 4.54   & 5.03  & 3.53   & 3.71    &  3.90    & 4.40  \\ \hline
    \end{tabular}
    \vspace{.3cm}
     \caption{Quantiles of the investment distribution for the random tax increase and the investment values for the deterministic increasing benchmark after $t=10$ and $t=15$ years, with and without rebate, for the filter technology with endogenous production. On average, investment is higher in the benchmark case.}
    \label{tab:investment_increase}

    \centering 
    \begin{tabular}{|l|cccc|cccc|}\hline
    &\multicolumn{4}{c|}{$t=10$}&\multicolumn{4}{c|}{$t=15$} \\ 
                        & 5\% & mean  &95\%  & benchmark &5\% & mean  &95\%  & benchmark   \\ \hline
        rebate    & 3.51  & 4.29  & 4.65   & 4.70    & 2.86     & 3.73      & 4.14  & 3.82\\
        no rebate   & 3.18  & 4.01   & 4.40   & 4.44    & 2.60     & 3.24      & 3.60  & 3.54 \\ \hline
    \end{tabular}
    \vspace{.3cm}
     \caption{Quantiles of the investment distribution for the random tax reversal and the investment values for the constant benchmark after $t=10$ and $t=15$ years, with and without rebate, for the filter technology with endogenous production. On average, investment is higher in the benchmark case.}
    \label{tab:investment_reversal}
\end{table}


\subsection{Two technologies with r tax risk} \label{sec:experiments-two-tech}

This section is dedicated to the analysis of some numerical experiments for the setup described in Example~\ref{subsec:two_technology}, where the producer can invest into a green production technology in addition to a brown one. 
In these experiments we assume the following form for the cost function: 
\begin{equation}\label{eq:numerics2tech}
C(q,x,y, \tau)=  (c_b +e_b\tau) Q_b\left(\left(q-P_g(x)\right)^+\right)\,,
\end{equation}
where $c_b=1$, $e_b=1$, $Q_b(q)=q^{3/2}$, $P_g(x)=p_g(x-\bar x)^+$. Here  $\bar x=20$ represents an initial expenditure that is necessary before the green investment is actually able to produce electricity such as the cost of buying land for solar farms or investments needed  to connect a solar park to   the grid.\footnote{To avoid numerical issues a smooth version of the function $P_g$ was used.}. We   set the productivity parameter to $p_g=0.2$ and we fix the maximum production capacity at $q^{\max}=10$. 
We moreover fix the following parameters: 
$T=15$ years, $h(x)=(0.7x)^+$, $\delta=0.02$, $\sigma=0.2$, $r=0.04$, $\kappa=0.5$. In addition, in the following experiments we assume that selling price of electricity is constant and equal to $p=2.1$, whereas the production $q$ is endogenous.

\subsubsection*{Tax risk.}
Similarly as in the case of the filter technology we model the tax process as a Markov chain with two states and we consider the same models for the tax dynamics, namely the tax increase and the tax reversal. In this section we let   $\tau^2 =1$. 
In addition we  consider a rebate of the form $\nu_0(q)\tau= e_b Q\left(\alpha q\right)\tau$, for $Q(q)=q^{3/2}$ and different values of $\alpha$. In particular $\alpha=0$ corresponds to the case where no rebate is enforced and  $\alpha=\frac{1}{2}$ means that a rebate is applied. Put in other words, rebate exceeds tax payments  as soon as the producer produces more than the fraction $1-\alpha$ of the total output using green technology. 

In Figure \ref{fig:Istar} we plot trajectories of optimal investments  with and without rebate. The left panel corresponds to tax increase and the right panel to tax reversal. For both tax trajectories, there is only  a moderate reaction of investment to tax switch, i.e. a slight modification of the slope of the investment trajectory when moving from the white area to the grey area and vice versa. Note that in case of the two technologies the producer has an incentive to invest into the green technology even for $\tau \equiv 0$, since the green technology has zero marginal cost. Hence  the impact of carbon taxes on investment  is smaller than for the filter technology   where,  without taxes, there is no economic incentive for investing into abatement technology.

Rebate is beneficial for investment for both tax models.  This effect is way more pronounced in the case of the random tax increase. We offer the following explanation. After the tax rise the producer benefits substantially from a high investment level since he  obtains a higher revenue from the  green electricity produced (market price plus rebate). Moreover, the producer knows that the tax rate will not return to $\tau^1 =0$ in the future, so that he can enjoy this high revenue for a longer period.

\begin{figure}[ht]
  \centering
  \includegraphics[width=.48\textwidth]{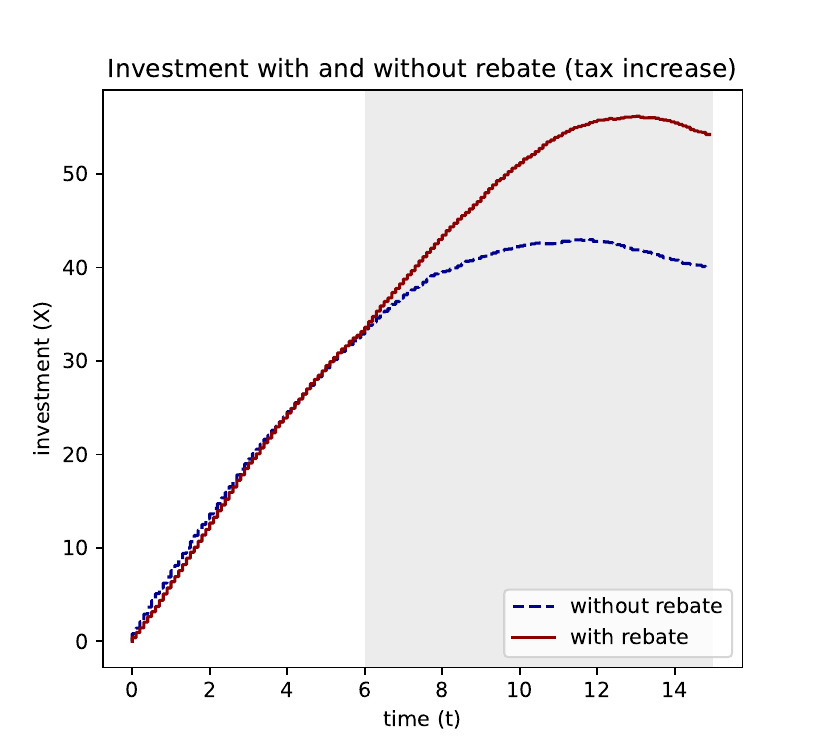}
  \includegraphics[width=.48\textwidth]{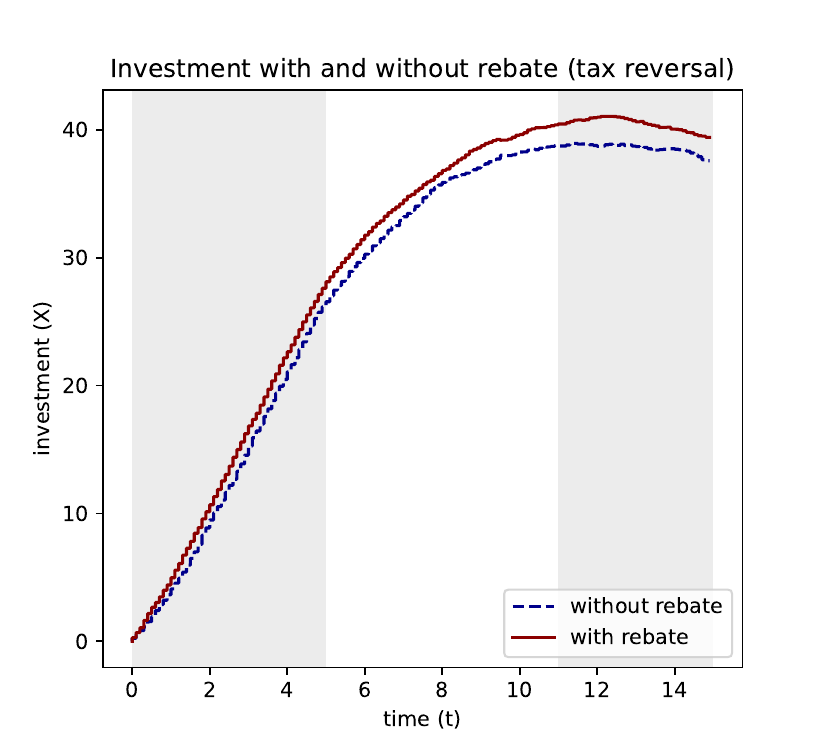}
\caption{Trajectory of the optimal investment $X$ for the case with rebate (solid red line) and without rebate (dashed blue line) for the cost function from the two-technologies example. The left panel corresponds to the random tax increase, the right panel to the  random tax reversal.} \label{fig:Istar}
\end{figure}



Finally, similarly as  in the example of the filter technology, we provide a comparison of the average investments,  the $5\%$ quantile and the $95\%$ quantile of the investment distribution   with and without rebate after 10 and 15 years. We moreover  report the values for the  benchmark cases which are computes using \eqref{eq:benchmark-increasing} and \eqref{eq:benchmark-constant}. The numbers for the tax increase are reported in Table~\ref{tab:emissions_increase_2tech} and for the tax reversal in Table~\ref{tab:emissions_reversal_2tech}. Consistently with the results for the single trajectory, the quantiles of the investment distribution under  rebate are always  higher than the quantiles without rebate, with a more pronounced  effect in the tax increase scenario. This holds also for the values obtained under the benchmark tax scenarios.
The two technology example confirms that the benchmarks perform better (stipulate more investment) than the average investment with random tax rates.  

\begin{table}[ht]
    \centering 
    \begin{tabular}{|l|cccc|cccc|}\hline
    &\multicolumn{4}{c|}{$t=10$}&\multicolumn{4}{c|}{$t=15$} \\ 
                        & 5\% & mean  &95\%  & benchmark &5\% & mean  &95\%  & benchmark   \\ \hline
        rebate    &  45.18  & 57.35  & 64.16  &  58.87   & 45.28    &   58.47   &  62.65  & 62.19 \\
        no-rebate    & 40.74  & 43.48  & 45.19  & 43.36   &  39.14   &  41.48    &  43.08 & 41.79 \\ \hline
    \end{tabular}
    \vspace{.3cm}
     \caption{Quantiles of the investment distribution for the random tax increase after $t=10$ and $t=15$ years. The linear increasing benchmark tax $\bar \tau (t)=bt$, $b=0.0985$ is computed as in Chapter \ref{sec:experiments-filter}.}
    \label{tab:emissions_increase_2tech}
\end{table}

\begin{table}[ht]
    \centering 
    \begin{tabular}{|l|cccc|cccc|}\hline
    &\multicolumn{4}{c|}{$t=10$}&\multicolumn{4}{c|}{$t=15$} \\ 
                        & 5\% & mean  &95\%  & benchmark &5\% & mean  &95\%  & benchmark   \\ \hline
        rebate    & 35.68  & 40.15  & 44.24  & 42.07   & 35.31    &    39.98  &  44.75 & 42.08 \\
        no-rebate    &  33.89 & 38.09  & 40.79  & 39.78   &   33.66  &   37.03   &  39.40 & 38.49 \\ \hline
    \end{tabular}
    \vspace{0.3cm}
    \caption{Quantiles of the investment distribution for the random tax reversal after $t=10$ and $t=15$ years. The constant benchmark tax $\bar \tau =0.565$, is computed as in Chapter \ref{sec:experiments-filter}.}
\label{tab:emissions_reversal_2tech}
\end{table}

\subsection{Two technologies with  tax uncertainty} \label{sec:tax_uncertainty_numerics} 

Next we report results from numerical experiments for the stochastic differential game where  tax rates are endogenously determined.  We work in the context of Example~\ref{subsec:two_technology} (the example with two technologies),  and we use the same parameters as in Section~\ref{sec:experiments-two-tech} except that we now work wit $T=10$.      We assume that tax rates take value in the interval $[0.5, 1.5]$ and we fix the \emph{most plausible} tax rate as  $\bar\tau \equiv 1$. The tax rebate is given by $\nu_0(q)\tau= e_b Q\left(\alpha q\right)\tau$, for  $ \alpha \in \{0, 0.5\}$, and the penalization for deviating from $\bar \tau$  by  $\nu_1 (\tau-\bar \tau)^2$, where $\nu_1 \in \{1,20\}$. The equilibrium  output $\hat q (x)$ and the equilibrium tax rate $\hat \tau(x)$ for this setup are discussed in Section~\ref{subsec:property-saddlepoint}, see in particular Figure~\ref{fig:saddle}. 

In Figure \ref{fig:Iaverage_2tech_uncertainty} we  plot the average investment $\mathbb{E}[X_t]$  under different values for  rebate and  penalization. The left panel corresponds to the case of high uncertainty ($\nu_1=1 $),  the right panel  corresponds to the  case of  low uncertainty  ($\nu_1 =20$). 
\begin{figure}[ht]
  \centering
  \includegraphics[width=.48\textwidth]{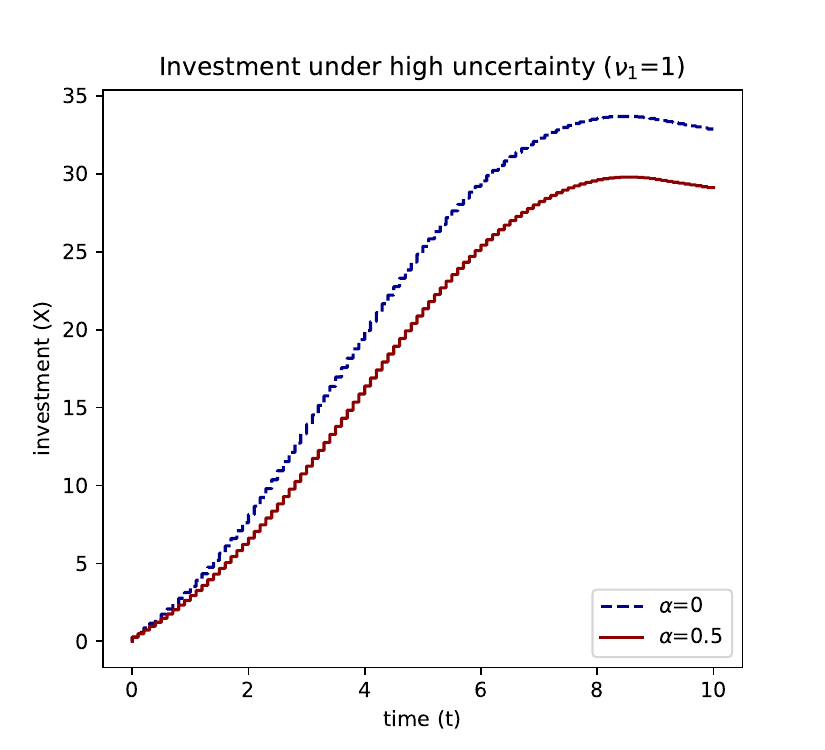}
  \includegraphics[width=.48\textwidth]{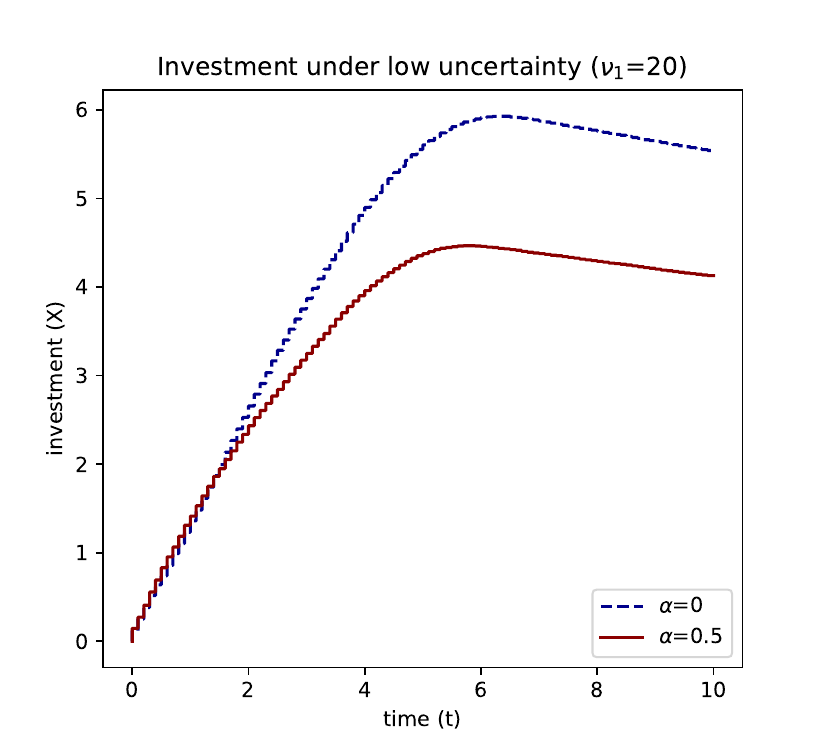}
\caption{Average investment $\mathbb{E}[I_t]$ under tax uncertainty  for  different values for  rebate and  penalization. The left panel corresponds to the case of high uncertainty ($\nu_1=1 $),  the right panel  corresponds to the  case of  low uncertainty  ($\nu_1 =20$). 
} \label{fig:Iaverage_2tech_uncertainty}
\end{figure}

 In this case we see that results from the tax risk paradigm are reversed. First, average investment for  high uncertainty   is  substantially higher  than for low uncertainty. This is due to the fact that  the equilibrium tax rate for $\nu_1 =1$  is higher than the equilibrium tax rate for $\nu_1 =20$, see Figure~\ref{fig:saddle}.\footnote{For $\alpha =0.5 $ this is true only for $x< 40$ but this is the relevant range to incentivise the buildup of green production capacities.} Indeed, higher tax rate generates more investment, so that  high  uncertainty is   beneficial from a societal point of view. Moreover, under tax uncertainty rebate reduces investment whereas in the tax risk case  a rebate led to an increase in investment.  The reason for this  difference is that the introduction of a rebate leads to a  lower equilibrium tax rate in the game between producer and opponent (see  again Figure~\ref{fig:saddle}).


\section{Conclusion}\label{sec:conclusion}

In this paper we analyzed the impact of  randomness in carbon tax policy on  the investment strategy  of a stylised  profit maximising electricity producer,  
who has to pay carbon taxes and decides on investments into  technologies for the  abatement of $\CO$ emissions. Adding to  the existing literature, we studied  a framework where  the investment in  abatement technology is divisible, irreversible and subject to transaction costs. We considered two approaches for modelling the randomness in taxes. First we assumed a precise probabilistic model for the tax process, namely a pure jump Markov process (so-called tax risk), which leads to a stochastic control  problem for the investment strategy.  Second, we  analyzed the case of  an {uncertainty-averse} producer who uses a differential game to decide on optimal production and investment.   We carried out a rigorous mathematical analysis of the producer's optimization problem and of the associated nonlinear PDEs (the HJB equation in the case of tax risk and the Bellman-Isaacs equation in the case of the stochastic differential game). In particular, we gave conditions for the existence of classical solutions in  both cases. Numerical methods were used  to analyze quantitative properties of the optimal investment strategy.  

Our experiments show  that under tax risk, the firm is typically less willing to invest into  abatement technologies than in a corresponding benchmark  scenario with a deterministic tax policy. Moreover, if a tax policy  that is not credible (i.e.~producers  are  not convinced that an announced tax increase will actually be implemented or they expect  that a high tax regime will be reversed soon), then it is substantially less effective. This  supports the widely held belief  that randomness in carbon taxes is in general  detrimental for climate policy.  These findings  have the following implications: 
a  climate tax  policy which is very mild initially and which postpones tax increases to random future time points may delay necessary  investment in green technology. On the other hand  a policy which is too stringent initially may generate strong  political pressure to revert to lower taxes,   which would be counterproductive for reducing carbon pollution.

Surprisingly, we found  that  under tax uncertainty  results are reversed. In a scenario with high uncertainty the  producer  invests more than under low uncertainty were taxes are almost deterministic,   so that an increase in uncertainty is beneficial from a societal point of view.  This is an  interesting observation,  which shows that the paradigm used to model the decision making process of the producer is a crucial determinant for the impact of randomness in climate policy.  
It is beyond the reach of this paper to make a scientific judgement as to which  of the two paradigms (risk or uncertainty) comes closer to the real decision making of investors and it is interesting to investigate the difference further. Intuitively,  we believe that the recommendations from the tax risk case are  more relevant for climate  policy.

\appendix

\section{Details on the numerical methodology.  }\label{appendix:numerical_methods}

For the numerical experiments in Section~\ref{sec:numerics}, we implemented the deep splitting method that was proposed  by \citet{beck2021deep} and extended to  partial integro-differential equations (PIDEs) by \citet{frey2022deep}. This approach uses deep neural networks to approximate the solution of a PIDE together with the gradients. Hence, we are able to compute the value function for the considered stochastic control problems and determine investments in green technology accordingly. In this section  we  present the basic idea of the algorithm. 
We consider a PIDE of the following form
\begin{align}\label{eq:PIDE}
	\begin{cases}
	u_t(t,z)+\mathcal{L} u (t,\psi)  = f\big(t,\psi,u(t,\psi), \partial_\psi u(t,\psi)\big)    & \text{on } [0,T) \times \R^n \, , \\
	u(T,\psi)=g(\psi) &    \text{on } \R^n  \,. 
	\end{cases}
\end{align}
Here $n=d+2$, $\psi=(x,y,\tau)\in \mathbb{R}^{n}$, $\partial_\psi u$  is the gradient of $u$ with respect to the space variable, $u_t$ the derivative with respect to the time variable, and
\begin{align*}
\mathcal{L} u(t,\psi) &:=  b(t,\psi) \cdot \partial_\psi u (t,\psi)  +\frac{1}{2} \sum_{i,j=1}^{n} (\Sigma \Sigma^\top)_{ij}(t,\psi)   u_{\psi_i\psi_j}(t_{\psi_i\psi_j},\psi) 
\\ & \qquad + \int_{\R} u(t,\psi+\widetilde{\Gamma}(t,\psi,z))-u(t,\psi)  \,  m(\ud z) \,,  \nonumber 
\end{align*}
where $b(t,\psi)=(-\delta x, \alpha(t,y), 0)^\top\in \mathbb{R}^n$, $\Sigma(t,\psi)\in \mathbb{R}^{n\times n}$ has components $\Sigma_{1,1}(\psi)=\sigma$, $\Sigma_{i,j}(t,\psi)=\alpha_{i-1,j-1}(t,y)$ for $i,j=2\dots,n-1$ and all other components  equal to zero, and $\widetilde{\Gamma}(t, \psi,z)\in \mathbb{R}^{n}=e_n  \Gamma(t,\psi, z)$, where $e_n$ is the $n$-th standard vector in $\mathbb{R}^n$. 
Next, we consider an auxiliary process, denoted as $\Psi$, whose dynamics correspond to the generator $\mathcal{L}$,
\begin{align}
\Psi_t&={ \Psi}_0+\int_{0}^{t} b( \Psi_s)  \, \ud s + \int_{0}^{t} \Sigma( \Psi_{s})  \, \ud \widetilde{W}_s + \int_{0}^{t} \int_{\R^d} \widetilde{\Gamma}( s,\Psi_{s-},z) {N}(\ud s, \ud z) \,.  \label{eq:FBSDEJ_X}
\end{align}
Specifically, in our context, $\Psi_t=(X^0_t, Y_t, \tau_t)$, where $X^0$ is the uncontrolled version of the process $X$, i.e. for $\gamma_t=0$.
The first step of the considered numerical algorithm is to  divide the time horizon into $N$ equidistant grid points $0=t_0 < t_1 < \dots < t_N=T$, where each interval is $\Delta t := 1/N$. Then, we discretize the process $\Psi$ using a method such as the Euler-Maruyama scheme along the given time grid. This discretization yields approximations for $\Psi_{t_i}$ at each time step $t_i$. We denote these approximation points as $\widehat{\Psi}_{t_i}$.  For the solution $u$ we consider a BSDE representation. Given that the PIDE admits a classical solution, we can apply It$\hat{\text{o}}$'s formula and express the solution as
\begin{align}
u({t_{i}},\Psi_{t_{i}})&=u({t_{i+1}},\Psi_{t_{i+1}})-\int_{t_{i}}^{t_{i+1}} \!\!\! f(s,\Psi_s,u(s,\Psi_s) ,\partial_\psi u(s,\Psi_s) ) \, \ud s - \int_{t_i}^{t_{i+1}}\!\!\! \Sigma(s, \Psi_s)^\top \partial_{\psi} u(s,\Psi_s) \,  \ud \widetilde{W}_s \\&- \int_{{t_{i}}}^{{t_{i+1}}} \int_{\R^d} u(s,\Psi_s+\widetilde{\Gamma}(s,\Psi_{s-}, z))-u(s,\Psi_{s-}) \, ({N}(\ud s, \ud z) - m(\ud z) \ud s) \, ,  \label{eq:FBSDEJ_Z}
\end{align}
Both the integral with respect to the Brownian motion and the integral with respect to the compensated jump measure are martingales (assuming sufficient regularity of $u$). Taking conditional expectations  leads to
\begin{align}
u(t_i,\Psi_{t_i}) = \mathbb{E} \Big [u(t_{i+1},\Psi_{t_{i+1}} )-\int_{t_i}^{t_{i+1}} f(s,\Psi_s,u(s,\Psi_s) ,\partial_\psi u(s,\Psi_s) ) \, \ud s \big | \Psi_{t_i}\Big]\, .  \label{eq:FBSDEJ_Z_2}
\end{align}
The discretization  allows us to approximate the integral term in the conditional expectation by $f(t_{i+1}, \widehat \Psi_{t_{i+1}},u(s,\widehat{\Psi}_{t_{i+1}}) ,\partial_\psi u(t_{i+1},\widehat{\Psi}_{t_{i+1}})) \Delta t$. Using the $L^2$-minimality of conditional expectations we represent $u(t_i,\Psi_{t_i})$ as the unique solution of the minimization problem over all $C^1$ functions
\begin{align*}
    \min_{U \in C^1} \,  \mathbb{E}_{t_i} \Big[ \big( U -  u(t_{i+1},\widehat{\Psi}_{t_{i+1}} )+f(t_i,\widehat{\Psi}_{t_{i+1}},u(t_{i+1},\widehat{\Psi}_{t_{i+1}} ) ,\partial_\psi u(t_{i+1},\widehat{\Psi}_{t_{i+1}} )) \Delta t  \big )^2\Big ]
\end{align*}
This minimization problem serves as a loss function for deep neural networks in the deep splitting algorithm, and  the algorithm can be summarized as follows.

\textit{Deep splitting algorithm.} Fix a class $\mathcal{N}$ of $C^1$ functions  $\mathcal{U}:\R^d \to \R$ that are given in terms of neural networks with fixed structure.    
 Then the  algorithm proceeds by backward induction  as follows.
\begin{enumerate}
	\item Let $\mathcal{\widehat U}_N=g$. 
	\item For $i=N-1,\dots,1,0,$  choose
	 $ \mathcal{\widehat U}_i $ as minimizer of the  loss function
$L_i \colon \mathcal{N} \to \R$,
	\begin{align}\label{eq:loss_fun}
	\mathcal{U} \mapsto \mathbb{E}\left[
 \Big| \mathcal{\widehat U}_{i+1}(\widehat{\Psi}_{t_{i+1}})  -\mathcal{U}(\widehat{\Psi}_{t_{i}}) -
    \Delta t \, f\Big(t_{i},\widehat{\Psi}_{t_{i+1}},\mathcal{\widehat U}_{i+1}(\widehat{\Psi}_{t_{i+1}}), D_x \mathcal{\widehat U}_{i+1}(\widehat{\Psi}_{t_{i+1}}) \Big)\Big|^2 \right] .
	\end{align}
	\end{enumerate}

Specifically, to address the numerical solution of this problem in our case studies, we generate simulations of trajectories for the processes $\Psi$. 
These simulations were carried out over the time interval $[0,15]$, discretized into 150 equally spaced time points (that is $N=150$ intervals). The inherent non-linearity of this problem is represented by the function:
\[
f(t,x, y, \tau, u_\psi) =\Pi^*(\psi) + \frac{\left((\partial_{\psi_1}u-1)^+\right)^2}{4\kappa} = \Pi^*(x, y,\tau) + \frac{\left((\partial_x u-1)^+\right)^2}{4\kappa}.
\]
We use deep neural networks with 2 hidden layers, each containing 40 nodes. In total, each experiment involves 150 networks. The neural networks are initialized with random values using the Xavier initialization scheme. We employ mini-batch optimization with a mini-batch size of $M = 10,000$, incorporating batch normalization. The training process spans 10,000 epochs, and the loss function is minimized through the Adam optimizer. The learning rate starts at $0.01$ and with a decay of $0.1$ every 4,000 steps. The activation function for the hidden layers is the sigmoid function, while the output layer uses the identity function. 

An advantage of this methodology is flexibility. The approach allows for effortless dimensionality adjustments in the state process $\Psi$ or modifications of its dynamics, with the only necessary adaptation being the Euler-Maruyama scheme for $\Psi$. For further details on deep splitting algorithms for general nonlinear PIDEs we refer to \citet{frey2022deep}.

\section{Some discussions and proofs for the tax risk setting}\label{appendix:tech_results}

In this section we present various technical results that are related to the characterization of the value function as  classical  solution of the HJB equation~\eqref{eq:HJB-bounded}. 

\subsection{Comments and extensions of Lemma \ref{lem:lipschitz}} \label{appendix:extension}
We now make few comments on possible extensions of the result stated in Lemma \ref{lem:lipschitz}, as anticipated in Remark \ref{rem:comments}. 

\noindent 1. \emph{Maximum capacity expansion.} In some examples it may make sense to assume that investment can expand maximum capacity. In such case a similar argument as in the proof of lemma \ref{lem:lipschitz}-(i) can be used to get regularity of the function $\Pi^*$. How to do that is briefly outlined next. 
If the maximum capacity depends on the investment level, i.e. $q^{\max}(x)$, for some Lipschitz continuous, increasing and bounded function, the above arguments can be extended. Indeed, in this case we have   
\begin{align}
&\left|\Pi^*(x^1, y^1, \tau^1) - \Pi^*(x^2,  y^2, \tau^2)\right|\\
&=\left|\max_{q \in[0,q^{\max} (x^1)]}\Pi(x^1, y^1, \tau^1, q)- \max_{q \in[0,q^{\max}(x^2)]} \Pi(x^2, y^2, \tau^2 q)\right|\\
& \le \left|\max_{q \in[0,q^{\max} (x^1)]}\Pi(x^1, y^1,\tau^1, q)- \max_{q \in[0,q^{\max}(x^1)]} \Pi(x^2,y^2,\tau^2, q)\right|\\
&+ \left|\max_{q \in[0,q^{\max} (x^2)]}\Pi(x^2, y^2,\tau^2, q)- \max_{q \in[0,q^{\max}(x^1)]} \Pi(x^2, y^2,\tau^2, q)\right|.
\end{align}
In the last expression, the first term is estimated exactly as in the proof of Lemma \ref{lem:lipschitz}-(i). In the second term, Lipschitzianity in $x$ is proved using Lipschitzianity of the function $q^{\max}(x)$.   

\noindent 2. \emph{Concavity of the value function.} 
If $\Pi^*$  and $h$ are concave in $x$, then, it can be proved that $V $ is also concave $x$.
Before going to the proof of this result, we highlight that an example where $\Pi^*$ is concave arises, for instance if $\Pi(t,x,y,\tau, q)$ is concave in $x$ and $q^*$ is a fixed quantity.
Indeed, 
the function $\Pi^*(t,x,y,\tau)$ 
is the result of an optimization and hence not an input variable of our model. This implies in particular that we cannot simply impose concavity, but we need to verify it, and, in general, even if $\Pi$ is concave, the supremum over $q$ may not be so. 

To establish concavity of the value function one can follow the steps below. We let for simplicity $t=0$.  Consider $X_0^1, X_0^2 >0$ and strategies $\bgamma^1, \bgamma^2 \in \mathcal{A}$. Denote by $X^j$, $j=1,2$, the investment process with initial value $X_0^j$ and strategy $\bgamma^j$ and let for $\lambda \in [0,1]$, $0 \le t \le T$,   $\bar X_t = \lambda X_t^1 + (1-\lambda) X_t^2$. Then it is easily seen that
$$ \ud \bar X_t = \lambda \gamma_t^1 + (1-\lambda) \gamma_t^2 - \delta \bar X_t \ud t + \sigma \ud W_t$$
so that $\bar X$ is the investment process corresponding to the strategy $\bar \bgamma \lambda \bgamma^1 + (1-\lambda) \bgamma^2$ with initial value $\bar X_0$ (Here we use that the dynamics of $X$ are linear). Concavity of $\pi^*$  and $h$ now imply that
$$ J(0,\bar X_0, y,\tau, \bar \bgamma) \ge \lambda J(0, X_0^1 , y,\tau,  \bgamma^1) + (1-\lambda) J(0, X_0^2 , y,\tau,  \bgamma^2)\,.$$
Concavity of $V$ follows from this inequality,  if we choose $\bgamma^j$  as an $\varepsilon$-optimal strategy  for the problem with initial value $X^j_0$.

\subsection{Proof of Theorem \ref{thm:class_sol}}
From Proposition \ref{thm:viscosity}, the function $ V(t, x, y, \tau)$ is Lipschitz continuous in $(x,y)$, H\"older in $t$ and the unique viscosity solution  of the PIDE
\begin{align}
 & v_t (t, x, y, \tau) + \Pi^*(x, y, \tau)  + \int_{\mathcal{Z}} v(t, x,y,  \tau + \Gamma(t,y,\tau, z))m(\ud z)\\
 & \quad+ \sum_{i=1}^d \beta_i(t, y) v_{y_i}(t, x, y, \tau)+  \frac{\sigma^2}{2}v_{xx}(t, x, y,  \tau) +  \frac{1}{2}\sum_{i,j=1}^d \mathfrak{S}_{ij}(t,y) v_{y_iy_j}(t, x, y, \tau) \\
  & \quad + \sup_{0\le \gamma\le \bar \gamma} (\gamma (v_{x}(t, x, y, \tau) -1) -\kappa \gamma^2 ) - \delta x v_x (t, x, y, \tau)= \left(r+m(\mathcal{Z})\right) v(t, x,y, \tau),
\end{align}
with the terminal condition $v(T,x,y,\tau)= h(x)$.
For  fixed $\tau$ we define the function $f^\tau(t, x, y):=\int_{\mathcal{Z}} V(t, x, y, \tau + \Gamma(t,y, \tau,z))m(\ud z) + \Pi^*(x,y, \tau)$.
Then for every fixed $\tau$,  $V^\tau (t,x,y) := V(t, x,y, \tau)$ is a viscosity solution of the equation
\begin{align}  
& u_t (t, x, y)  + \sum_{i=1}^d \beta_i(t, y) u_{y_i}(t, x, y) +  \frac{\sigma^2}{2}u_{xx}(t, x,y) +  \frac{1}{2}\sum_{i,j=1}^d \mathfrak{S}_{ij}(t,y) v_{y_iy_j}(t, x, y, \tau) \\
  &\quad +  \sup_{0\le \gamma\le \bar \gamma} (\gamma (u_{x}(t, x,y) -1) -\kappa \gamma^2 )- \delta x u_x (t, x,y) +  f^{\tau}(t, x, y)= R u(t, x,y), \label{eq:parabolic}
\end{align}
with $u(T, x, y)=h(x)$ and $R = r + m(\mathcal{Z})$. Note that this is a quasilinear parabolic  PDE since there are no non-local terms and  since for all $p \in \mathbb{R}$,
$$\sup_{0\le \gamma\le \bar \gamma}\{ p\gamma - \gamma -\kappa \gamma^2\}=\left\{\begin{aligned} 0 \qquad & \quad \mbox{ if } p<1\\\frac{[(p-1)^+]^2}{4\kappa} &\quad \mbox{ if } 1\le p \le 2 \kappa +1 \\ \kappa \bar \gamma^2 \quad &\quad \mbox{ if } p > 2 \kappa +1 \end{aligned}\right.$$
Our goal is to show that this PDE has a classical solution which coincides with $V^\tau$. We proceed  in several steps.

\emph{Step 1.} Fix $K >0$ and define the set $Q_K=[0,T]\times B_K$, where $B_K=\{\mathbf{x} \in \mathbb{R}^{d+1}: \|\mathbf{x}\|^2\le K^2\}$,  and let $\mathcal G_K=\{T\}\times B_K \cup [0,T) \times S_K$ where $S_K=\{\mathbf{x} \in \mathbb{R}^{d+1}: \|\mathbf{x}\|^2 = K^2\}$. Consider the terminal boundary value problem consisting of the PDE \eqref{eq:parabolic} and the boundary condition $u = V^\tau$ on  $\mathcal G_K$. We now use  Theorem 6.4 in \citet[Ch. 5]{bib:lsu-68} to show that this terminal boundary value problem has a classical solution that is moreover smooth on the interior of  $Q_K$. For this we  formulate \eqref{eq:parabolic} as a parabolic equation in divergence form. We define for 
$y=(y_1, \dots, y_d)$, $p_2=(p_{2,1}, \dots, p_{2,d})$ the functions
\begin{align}
A(t, x, y, u, p_1, p_2)= & \sum_{i=1}^d \beta_i(t, y) p_{2,i} + \sup_{0\leq \gamma \leq \bar \gamma } \{\gamma (p_1-1) -\kappa \gamma^2\}- \delta x p_1    - R u +  f^\tau(t, x, y)\\
a(t, x, y, u, p_1, p_2) = & A(t, x, y, u, p_1, p_2)+ \sum_{i,j=1}^d \partial y_i \mathfrak{S}_{ij}(t,y) p_{2,j} \\
a_1(t, x, y, u, p_1, p_2) = & \ \frac{\sigma^2}{2} p_1\\
a_{2,i}(t, x, y, u, p_1, p_2) = & \frac{1}{2}\sum_{j=1}^d  \mathfrak{S}_{ij}(t,y) p_{2,j}, \quad i =1, \dots p
\end{align}
Then  \eqref{eq:parabolic} can be written in divergence form as in equation (6.1) of  \cite[Chapter 5]{bib:lsu-68}:
\begin{align}
\partial_t u & + {\partial_x} a_1(t, x, y, u, u_x, u_y)  +  
 \sum_{i=1}^d \partial y_i a_{2,i}(t, x, y, u, u_x, u_y) -  a(t, x, y, u, u_x, u_y)=0.
\end{align}
Note that the signs differ  from those in \cite{bib:lsu-68} since we are dealing with a \emph{terminal} value condition.

Next we show that the assumptions of Theorem 6.4 in \cite[Ch. 5]{bib:lsu-68} are satisfied on the domain $Q_K$. Note first that 
the set $S_K$ is the boundary of the $d+1$-dimensional circle so it is smooth and hence satisfies condition (A) (see \cite[page 9]{bib:lsu-68}).  
Moreover,  
\[
 A(t, x, y, u, 0, 0) u =   - \left(r+m(\mathcal{Z})\right) \  u^2 +  f^{\tau}(t, x, y) \ u \ge - b_1 u^2-b_2
 \]
for $b_1, b_2\ge 0$, since 
the functions  $f^{\tau}(t, x, y)$ are bounded on $Q_K$. To see the latter recall that  $f^{\tau}(t, x, y):=\int_{\mathcal{Z}} V(t, x, y, \tau+\Gamma(t,y,\tau, z))m(\ud z) + \Pi^*(x, \tau,y )$, and   $\Pi^*(x, \tau,y )$ and $V$ are bounded on the bounded set $Q_K$ respectively $Q_K \times [0,\tau^{\max}]$.  hence the inequality holds. That guarantees that condition a) of Theorem 6.4 \cite[Ch. 5]{bib:lsu-68} holds. 
Conditions (3.1), (3.2), (3.3), (3.4)  \cite[Ch. 5]{bib:lsu-68} are immediate. In particular, the condition $\sigma^2 >0$ and the strict ellipticality of $\mathfrak{S}(t,y)$ ensure that the crucial condition (3.1) holds.  Finally,  since $V(t,x,y, \tau)$ is a Lipschitz viscosity solution of the HJB equation, the boundary condition is Lipschitz, which in particular implies condition c) of Theorem 6.4 \cite[Chapter 5]{bib:lsu-68}. 
By applying Theorem 6.4 in \cite[Ch. 5]{bib:lsu-68}, we thus get that in any interior subdomain $Q_K$ the HJB equation has a classical solution $U^\tau(t,x,y)$ which coincides with $V^\tau(t, x, y)$ on the boundary $\mathcal G_K$.

\emph{Step 2.} Next we show that $U^\tau(t,x,y)=V(t, x, y, \tau)$ in the interior of $Q_K$ for every $K$ which allows to conclude that $V(t, x, y, \tau)$ is smooth in the interior of $Q_K$. 
To prove this we apply the comparison principle given by \cite[Corollary  8.1, Ch.5]{fleming2006controlled}. Note that inequality (7.1) on page 218 of the book is implied by in particular by Lipschitzianity of the functions $\alpha, \beta, \Gamma$ in $y$. 
Then we obtain that $U^\tau(t,x,y)=V(t, x, y, \tau)$ on $Q_K$. 

Since $K$ was arbitrary we finally  get that $V$ is smooth everywhere. Hence $V$ is also a classical solution of the HJB equation. \eqref{eq:HJB-bounded}, which concludes the proof.

\subsection{An example with strict viscosity solution}\label{app:viscosity}
In the following section we present an example illustrating that, in general, the value function may be  non-smooth and hence  a strict viscosity solution of the HJB equation. Specifically, we examine the cost function associated with the filter technology, assuming a fixed electricity price $\bar p$ and a fixed production quantity $\bar q$.  To present this example with minimal technical difficulties, we make certain assumptions. We set $r$ and $\delta$ to zero, take  the residual value as $h(X_T) = 0$, and assume deterministic tax rate equal to $\bar\tau>0$. Additionally, we adopt the  abatement technology  $e(x) = (1 - x)^+$ and assume no external variations in the investment level ($\sigma = 0$). This assumption is  crucial for our  our example, since for $\sigma >0$ the HJB equation has a classical solution by Theorem~\ref{thm:class_sol}   Section~\ref{subsec:classical}.

In this setting $X_t=X_0+\int_0^t\gamma_s \ud s$, and the value function is given by
\begin{align} \label{eq:def-tildeV}
V(t, x)&=\sup_{\bgamma \in \mathcal A}\mathbb{E}_t\left[\int_t^T\! \left(\bar p \bar q  - \bar q (\bar c + (1- X_s)^+ \bar \tau)-\gamma_s-\kappa \gamma_s^2 \right)\, \ud s \right] =:  \bar p \bar q - \bar q \bar c + \bar q \tilde V(t,x)\,\, \\
\intertext{where}
\label{eq:def-tildeV_2}
\tilde V(t,x) &= \sup_{\bgamma \in \mathcal A}\mathbb{E}_t\left[\int_t^T (1- X_s)^+ \bar \tau -\gamma_s-\kappa \gamma_s^2 \,\ud s \right]\,.
\end{align}
In the sequel we concentrate on $ \tilde V$. Note first that for  $x\ge 1$,   the optimal strategy is $\bgamma^* = 0$, since choosing $\gamma_s >0$ is costly but generates  no additional reduction in emissions. Therefore, $\tilde V(t,x) =0$ for $x \ge 1$. Below  we show that
\begin{equation}\label{eq:bound_on_tildeV}
\tilde V(t,x) \le  - (1-x) \big ( 1 \wedge (T-t)\bar \tau \big )\,,\quad x\le 1.
\end{equation}
Let $\tilde V_{x^-} (t,1)$ be the left derivative of $\tilde V (t,\cdot)$ at $x=1$. It follows that
$$\tilde V_{x^-} (t,1) = \lim_{h \to 0^+}  \frac{1}{(-h)} ( \tilde V(t,1-h) - \tilde V(t,1))  \ge  ( 1 \wedge (T-t)\bar \tau \big ).$$
Hence $\tilde V(t,\cdot) $ has a kink at $x=1$, and from equation \eqref{eq:def-tildeV_2}, we get that $V$ is  a strict viscosity solution of the HJB equation.  

\noindent Now we turn to the inequality \eqref{eq:bound_on_tildeV}. Obviously,
\begin{equation} \label{eq:no_tc}
\tilde V(t,x)  \le  w(t,x):= \sup_{\bgamma \in \mathcal A}\mathbb{E}_t\left[\int_t^T - (1- X_s)^+ \bar \tau -\gamma_s  \, \ud s \right]\,.
\end{equation}
Since in \eqref{eq:no_tc} transaction costs are zero,  the producer can push $x$ instantaneously to any level $x' >x$, incurring a cost of size   $x'-x$. It follows that for $x<1$, the  ``limiting optimal strategy'' in $\eqref{eq:no_tc}$ is to push the investment level to $1$ immediately at $t$, provided the resulting tax savings $\bar \tau (1-x) (T-t)$ exceed the cost $1-x$, and to choose $\bgamma \equiv 0$  otherwise. This gives
$$
u(t,x) = \begin{cases}
-(1-x) & \text{ if } \bar \tau (T-t) >1\,,\\
-(1-x) & \text{ if } \bar \tau (T-t) \le 1\,,
\end{cases}
$$
that is, $u(t,x) = -(1-x) \big ( 1 \wedge (T-t)\bar \tau \big )$,  $x\le 1$, which implies  \eqref{eq:bound_on_tildeV}.

\section{Differential game} \label{app:game}
\subsection{Proof of Lemma~\ref{lem:saddle-point}}
Define the compact and convex set $B:= [0,q^{\max}]
\times [\tau^{\min}, \tau^{\max}]$ and the function $F\colon B \to B$ by $F(q,\tau) = (q(\tau),\tau(q))'$.  Note that $q(\tau)$ and $\tau(q)$ and hence $F$ are continuous on $B$ (since $\partial_q C_0$ and $\partial_q C_1$ are strictly increasing and since $\nu_1 >0$). By \eqref{eq:char_saddle_point}, $(q^*,\tau^*)$ is a saddle point of $g$ if and only if it is a fixed point of $F$ on $B$. The existence of a fixed point of $F$ follows immediately from Brouwers fixed point theorem, which establishes the existence of a saddle point of $g$.     

For uniqueness note that the  pair $(q^*,\tau^*)$ is a saddle point if and only if $q^*$  satisfies the fixed point relation $q^* = q(\tau(q^*))$ and if $\tau^* = \tau(q^*)$. Define the mapping $\varphi \colon [0,q^{\max}] \to \R$ with 
$$\varphi(q):= p - \partial_q C_0(q) - (\partial_q C_1(q)  - \partial_q\nu_0 (q) )\tau(q).$$
By the FOC characterizing $q(\tau)$,  a solution   $q^* \in [0,q^{\max}]$ is a solution of  the  equation $q^* = q(\tau(q^*))$ if one of the following three conditions hold (i) $\varphi(q^*) =0$; (ii)  $\varphi(0) < 0$ , in which case  $q^*=0$; (iii) $\varphi(q^{\max}) > 0$, in which case  $q^*=q^{\max}$.  Below we show that $\varphi$ is strictly decreasing. It follows that there is at most one $q^* \in [0,q^{\max}]$ that fulfills (i), (ii) or (iii)  and hence at most one saddle point.

To show that $\varphi$ is strictly decreasing we first  we compute the derivative of $\varphi$ for those values of $q$ with $\tau(q) \in (\tau^{\min}, \tau^{\max})$. We get 
\begin{align*}
\partial_q \varphi (q) &= - \partial^2 _q C_0 - (\partial^2_q C_1(q) - \partial^2_q\nu_0 (q)) \tau(q) - (\partial_q C_1(q)  -\nu_0'(q) ) \partial_q \tau(q) \\&=    - \partial^2 _q C_0 - (\partial^2_q C_1(q) -\partial^2_q\nu_0(q)) \tau(q) - \frac{1}{2 \nu_1}(\partial_q C_1(q)  -\partial_q \nu_0(q) )^2\,,
\end{align*}
which is negative due to the assumptions on  $C_0$ $C_1$ and $\nu_0$. For values of $q$ where the constraints on $\tau$ bind we have $ \partial_q\tau(q) =0$ and 
$$\partial_q \varphi (q) = - \partial^2 _q C_0 -(\partial^2_q C_1(q) -\partial_q^2 \nu_0(q)) \tau(q) <0. $$ 
It follows that $\varphi$ is absolutely continuous with strictly negative  derivative and hence strictly decreasing.


\end{document}